\documentclass[bj]{imsart}

\RequirePackage{amsthm,amsmath,amsfonts,amssymb}
\usepackage{euscript}
\usepackage{todonotes}
\RequirePackage[numbers]{natbib}
\RequirePackage{mathrsfs}
\usepackage{subcaption}
\usepackage{verbatim}

\newcommand{\CBI}{{\operatorname{CBI}}}
\newcommand{\D}{\mathbb{D}}
\newcommand{\E}{\mathbb{E}}
\newcommand{\EE}{\mathscr{E}}
\newcommand{\eps}{\varepsilon}
\newcommand{\F}{\mathcal{F}}
\newcommand{\GG}{\mathcal{G}}
\newcommand{\hp}{\hat{\phi}}
\newcommand{\htu}{{\hat{\tau}^\uparrow}}
\newcommand{\htd}{{\hat{\tau}^\downarrow}}
\newcommand{\htds}{{\hat{\tau}^{\downarrow,-}}}
\newcommand{\htdl}{{\hat{\tau}^{\downarrow,+}}}
\newcommand{\tu}{\tau^{\uparrow}}
\newcommand{\td}{\tau^{\downarrow}}
\newcommand{\tds}{\tau^{\downarrow,-}}
\newcommand{\tdl}{\tau^{\downarrow,+}}
\newcommand{\Leb}{{\operatorname{Leb}}}
\newcommand{\M}{\mathscr{M}}
\newcommand{\n}{\mathfrak{n}}
\newcommand{\NN}{\mathcal{N}}
\newcommand{\Om}{\Omega}
\newcommand{\PD}{\operatorname{PD}}
\newcommand{\PR}{\mathbb{P}}
\newcommand{\Q}{\mathbb{Q}}
\newcommand{\QQ}{\tilde{\mathbb{Q}}}
\newcommand{\R}{\mathbb{R}}
\newcommand{\TIS}{\operatorname{SCBI}}
\newcommand{\w}{\omega}

\RequirePackage[colorlinks,citecolor=blue,urlcolor=blue]{hyperref}

\startlocaldefs
\theoremstyle{plain}
\newtheorem{thm}{Theorem}
\newtheorem{lem}{Lemma}
\newtheorem{prop}{Proposition}
\newtheorem{cor}{Corollary}
\newtheorem{claim}{Claim}
\theoremstyle{remark}
\newtheorem{definition}{Definition}
\newtheorem{rmk}{Remark}

\endlocaldefs

\begin{document}

\begin{frontmatter}
\title{Occupation Times for Time-changed Processes with Applications to Parisian Options}
\runtitle{Occupation times for Parisian Options}

\begin{aug}
\author[A]{\fnms{Joonyong} \snm{Choi}\ead[label=e1]{jy27@uw.edu}},
\and \author[A]{\fnms{David} \snm{Clancy, Jr.}\ead[label=e2,mark]{djclancy@uw.edu}}

\address[A]{University of Washington, Seattle, \printead{e1}}

\address{\printead{e2}}
\end{aug}

\begin{abstract}
Stochastic processes time-changed by an inverse subordinator have been suggested as a way to model the price of assets in illiquid markets, where the jumps of the subordinator correspond to periods of time where one is unable to sell an asset. We develop an excursion theory for time-changed reflected Brownian motion and use this to express the price of certain European options with Parisian barrier condition in terms of solutions of a time-fractional PDE. We provide a general description of the occupation measures of time-changed processes and use this to prove a Ray-Knight theorem for the occupation measure of a time-changed Brownian motion with negative drift. We also show that the duration of the excursions on finite time intervals obey a Poisson-Dirichlet distribution when a reflected Brownian motion is time-changed by an inverse stable subordinator.
\end{abstract}



\end{frontmatter}


\section{Introduction}

	It is well-known that functionals of Brownian motion or, more generally, L\'{e}vy processes appear in pricing economic assets. Often these processes have infinite variation and hence are incapable of handling situation where the price of an asset is constant for a period of time. Therefore changes to the stochastic model should be made when the underlying assets are 
	illiquid, i.e. the assets cannot be traded quickly. One attempt to overcome this shortcoming is to introduce a random time-change $E = (E(t);t\ge 0)$ which is the inverse of a subordinator \cite{HL.21,Magdziarz.09}. See also \cite{GMS.01, MRGS.00, SGM.00} for an approach involving discrete observations of prices and a limiting procedure, whose idea originates from \cite{MS.65} in physics. The governing Fokker-Planck equation for the non-time-changed stochastic process is often replaced with a partial differential equation (PDE) where the time derivative is replaced with some fractional time derivative (time-fractional PDE). See \cite{HL.21} for an example involving illiquid markets and see \cite{Chen.17} and references therein for a more general approach of using time-changed stochastic processes to solve time-fractional PDEs. These time-changed processes, or the governing time-fractional PDEs, have been used to study subdiffusive behavior in physics as well, where the motionless behavior stems from trapping of particles \cite{MS.19,MK.04}.
	
 	For concreteness, let us turn to the subdiffusive Black-Scholes model with zero interest rate. In the Black-Scholes model, the price of a stock is modeled by a geometric Brownian motion $X = (X(t);t\ge 0)$. In the subdiffusive model, corresponding to an illiquid market, one has to modify the stock price by introducing a time-change $E = (E(t);t\ge 0)$ which is the inverse of a strictly increasing L\'{e}vy process $S = (S(t);t\ge 0)$ independent of the process $X$. That is, in the illiquid market the stock price is modeled by process $X^* = (X^*(t);t\in[0,T])$ defined by $X^*(t) = X(E(t))$ where
	\begin{equation}\label{eqn:Edef}
	E(t) = \inf\{u: S(u)>t\}.
	\end{equation} 
	This model is considered in \cite{Magdziarz.09, MG.12} where the authors show that this model of an illiquid market is arbitrage-free but incomplete. One has to be careful about how an interest rate is incorporated - see \cite{MG.12}.
	
	Standard barrier options are among the most popular derivative contracts \cite{carr1995two}. Depending on its type, the value of the option depends on whether the underlying asset price has or has not hit some predefined price, the barrier, before the maturity date. Due to the one-touch nature of these options, influential agents can 
	manipulate the price of the underlying asset when the price is close to the barriers. To mitigate this problem Chesney et. al. \cite{CJY.97} introduced the so-called Parisian barrier option which depends on the occupation times of the price of underlying assets. 
	Instead of the value of the option depending on whether or not the underlying asset hits some predefined level, the underlying asset has to spend a predefined amount of time above or below the barrier before maturation. For example, the owner of a down-and-out option loses the option when the asset price $X$ remains below level $L$ for $D$ units of time and this has to occur before maturation at time $T$. Therefore, in order to understand how to price Parisian options, one needs to understand properties of the occupation measure and excursions of the process $X$.
	
	There is a vast amount of literature involving the occupation times of Brownian motion. In particular, the length of the longest excursion of a Brownian motion on some fixed (possibly random) time interval has been well-studied. See \cite{PY.97} for a survey and generalizations of this. This literature allowed the authors of \cite{CJY.97} to obtain some complicated, although explicit, formulas for the valuation of Parisian options using Laplace transforms. With this said, it appears that Parisian options for the time-changed processes is understudied. This paper is devoted to filling in some of the mathematical gaps in this literature. We develop an excursion theory for time-changed Brownian motion and then use this to price a down-and-out European Parisian option when the stock price is modeled by a time-changed geometric Brownian motion. We relate this price to solutions of time-fractional PDEs. We also develop some process-level descriptions relating the occupation measures for processes time-changed by inverse subordinators to those of the original non-time-changed process. In particular, we obtain an explicit Ray-Knight type theorem for a time-changed Brownian motion (with drift) killed upon hitting zero.

	\subsection{Main Results}
	
	
	\subsubsection{Excursion Intervals of $X^*$}
	
	Consider the situation where $X$ is a reflected Brownian motion and $S$ is an independent $\beta$-stable subordinator for $\beta\in(0,1)$. The process $S$ is uniquely characterized by its Laplace transform
	\begin{equation}\label{eqn:beta_Stable}
	    \E\left[ \exp(-\lambda S(t)) \right] = \exp(-ct\lambda^\beta), \qquad \forall t, \lambda>0
	\end{equation} for some constant $c>0$. See \cite[Chapter III]{Bertoin.96} for more information on subordinators. Since its L\'{e}vy measure $\nu(dy) = \tilde{c} y^{-1-\beta}1_{[y>0]}\,dy$ has infinite mass, $S$ is strictly increasing. 
	
	Before turning to our first theorem, let us briefly recall the Poisson-Dirichlet distribution $\PD(\alpha,\theta)$ as in \cite{PY.97}. Since our results only rely on $\PD(\alpha,0)$, we will focus on this case. Let $J_1\ge J_2\ge \dotsm >0$ be a rearrangement of all the jumps of $S$ on the interval $[0,1]$: 
	\begin{equation*}
		\{J_1,J_2,\dotsm\}  = \{S(t)-S(t-): t\in[0,1]\} \setminus \{0\}.
	\end{equation*} Consider the random element $V = (V_1,V_2,\dotsm)$ of the infinite simplex $\Delta_\infty:= \{(x_j)\in [0,1]^\infty: x_1\ge x_2\ge \dotsm \ge 0, \sum_{j} x_j = 1\}$ defined by $V_n = J_n/S(1) = J_n/ \sum_{i} J_i$. Then $\PD(\beta,0)$ is the law of the vector $V$ in $\Delta_\infty$.
	
	After developing some excursion theory for the time-changed reflected Brownian motion $X^*$, we obtain the following Theorem:
		\begin{thm}\label{thm:PD}
	Let $X$ denote a reflected Brownian motion, and let $X^*(t) = X(E(t))$ where $E$ is the right-continuous inverse of a $\beta$-stable subordinator defined by \eqref{eqn:Edef}. Then the ordered lengths of the relatively open intervals of $\{u\in[0,t]: X^*(u)>0\}$, say $U_1(t)\ge U_2(t)\ge\dotsm> 0$, are distributed as $t$ times a $\PD(\beta/2,0)$ random variable:
	\begin{equation*}
		\left(\frac{U_1(t)}{t},\frac{U_2(t)}{t},\dotsm\right) \overset{d}{=} \PD\left(\frac{\beta}{2},0\right).
	\end{equation*}
	\end{thm}

	   Observe that this conclusion does not depend on the particular value of $c$ in \eqref{eqn:beta_Stable}. While this gives the joint distribution of the lengths of all the excursions (appropriately reordered), it is restricted to simply the case of a stable subordinator and the excursions of a reflected Brownian motion at zero. More careful analysis would have to be done in order to handle the case of time-changing a Brownian motion (with drift) by a subordinator which is not stable or even when examining the excursions above a level $b>0$. 
	   
	\subsubsection{Price of an option}
	
	Recall the subdiffusive Black-Scholes model discussed above. The process $X^*$ is the solution to 
	\begin{equation}\label{eqn:agbm}
	dX^*(t) = X^*(t) \left( \mu dE(t) + \sigma d W(E(t))\right),\qquad X^*(0) = x>0,
	\end{equation} where $W$ is a standard Brownian motion. Here $E$ is defined by \eqref{eqn:Edef} for some strictly increasing subordinator $S$ is independent of $W$. This is defined on some filtered probability space $(\Om,\F,\{\F_t\}_{t\ge 0} ,\QQ)$ where $\F_t$ is the standard augmentation of $\sigma\left\{X^*(s);s\le t\right\}$ and $\QQ$ is the physical measure. That is, $W$ is a $\QQ$-Brownian motion. As is described in \cite{MG.12}, there exists a (not unique!) equivalent martingale measure $\Q$ under which $X^*$ is an $(\F_t)$-martingale.

	  A Parisian option depends on several parameters which we now fix: the barrier level $L$; an excursion duration $D$; and a maturation time $T$. Consider the first time $H_{L,D}^-(w)$ that a continuous path $w:\R_+\to \R$ stays strictly below level $L$ for duration $D$. That is 
	  \begin{equation}\label{eqn:HLD_def}
	  H_{L,D}^-(w) = \inf\{t: 1_{[w(t)< L]} (t- g_L^w(t))> D \},\qquad g_L^w(t) = \sup\{r\le t: w(r) = L\},
	  \end{equation} with the convention that $\sup\emptyset = 0$ and $\inf\emptyset  = +\infty$. Note that $g_{L}^{w}(t)$ is the last time the path $w$ hits level $L$. The down-and-in European Parisian option with payoff $\Phi$ is therefore priced as \begin{equation}\label{eqn:CdiForTimeChange}
	  V_{d,i}^*(x, T,L,D,\Phi) = \E_{\Q} \left[1_{[H^-_{L,D}(X^*)\le T]}\Phi(X^*(T)) \right].
	  \end{equation}
	  
	  The next theorem describes the price of a down-and-in Parisian option in terms of solutions to time-fractional partial differential equations when the barrier level $L$ is $x$. For more information on time-fractional PDEs, see Section \ref{sec:timeFracPDE}. See also equation \eqref{eqn:genCaputo} for the definition of the generalized time-fractional derivative $\partial_{t}^{S}$. 
	  \begin{thm}\label{thm:priceTheorem}
	  Consider the above set-up. Assume the L\'evy measure of $S$ has infinite mass and denote by $\phi$ the Laplace exponent of $S$ as in \eqref{eqn:LaplaceExponent}. Define the measure $\pi$ on $\R_+^2$ by $\pi(A\times B) = \int_{B} \PR(S(t)\in A) \frac{1}{\sqrt{2\pi t^3}}\,dt$. Let $\pi^+(A\times B) := \pi\left( (A\cap (D,\infty))\times B\right)$, $\pi^{-}(A\times B):=\pi^-((A\cap (0,D])\times B)$ and let $\mu = \pi^+(\R_+\times \R_+)<\infty$. 
	  	
	 Suppose that $\Phi:\R_+\to \R$ is smooth and vanishes in a neighborhood containing $x$, the initial stock price. Then a price of the down-and-in European Parisian option with barrier $L = x$ and payoff $\Phi$ is
	  \begin{equation*}
	  \begin{split}
	  V_{d,i}^*(x, T, x, D, \Phi)	  &= \int 1_{[s\le T-D]} e^{-\frac{\sigma^2}{8} v} u_1(0,0; T-s) m_1(ds,dv) \\
	  &\qquad + \int 1_{[D\le s \le T]} e^{-\frac{\sigma^2}{8} v} u_2(0,0;T-s)\,m_2(ds,dv),
	  \end{split}
	  \end{equation*}
	  where the Laplace transforms of the measures $m_j(ds,dv)$ are
	  \begin{equation*}
	 \begin{split} \int e^{-\lambda s - \theta v}\,m_1(ds,dv) &= \frac{\mu}{\displaystyle \mu + \sqrt{\frac{1}{2}(\phi(\lambda)+\theta)} +\frac{1}{2} \int_{\R_+^2} (1-e^{-\lambda s - \theta v}) \pi^-(ds,dv)},\\
	 \int e^{-\lambda s - \theta v}\,m_2(ds,dv) &= \frac{\displaystyle \frac{1}{2}\int_{\R_+^2} e^{-\lambda s - \theta v} \pi^+(ds,dv)}{\displaystyle \mu + \sqrt{\frac{1}{2}(\phi(\lambda)+\theta)} +\frac{1}{2} \int_{\R_+^2} (1-e^{-\lambda s - \theta v}) \pi^-(ds,dv)};
	 \end{split}
	  \end{equation*}
	  here $u_1(z,s;t)$ is the solution to the time-fractional PDE 
	  \begin{equation*}
	  \left\{\begin{array}{l}
	  \partial_t^S u = \left(\frac{1}{2} \frac{\partial^2}{\partial z^2} + \frac{1}{2z} \frac{\partial}{\partial z} + \frac{\partial}{\partial s}\right) u\\
	  u(z,s;0) = \exp\left\{-\frac{\sigma}{2} z - \frac{\sigma^2}{8} s\right\} \Phi\left( xe^{-\sigma z}\right)
	  \end{array}
	  \right.;
	  \end{equation*}
	  and $u_2$ is the solution to
	   \begin{equation*}
	  \left\{\begin{array}{l}
	  \partial_t^S u = \left(\frac{1}{2} \frac{\partial^2}{\partial z^2} +  \frac{\partial}{\partial s}\right) u\\
	  u(z,s;0) = \exp\left\{\frac{\sigma}{2} z - \frac{\sigma^2}{8} s\right\} \Phi(xe^{\sigma z})
	  \end{array}
	  \right..
	  \end{equation*}
	  \end{thm}
	  \begin{rmk}
	  	The requirement that $\Phi$ is continuous means that this result does \textit{not} give us the prices of the down-and-in Parisian calls or puts. We can recover the price of the Parisian call option with payoff function $(\cdot -K)_{+}$ by a limiting procedure provided that $K\ne x$. We \textbf{are not claiming} there is a corresponding representation in terms of time-fractional PDE, which would require additional work.
	\end{rmk}

	\begin{rmk}
		The requirement that $\Phi$ vanishes near the initial stock price is not the most general statement we can make. It can also be stated that $z\mapsto e^{-\frac{\sigma}{2}{z}} \Phi(xe^{-\sigma z})$ belongs to the domain of the infinitesimal generator of a 2-dimensional Bessel process which can be found in Section 2.7 of \cite{IM.74}.
	\end{rmk}
	  
	 \subsubsection{Ray-Knight type theorems}
	   Theorem \ref{thm:PD} above gives some information on lengths of excursions while it does not give us detailed information on the occupation measures of $X^*$. In Section \ref{sec:occupationChange} we study the effect that the time-change has on the occupation measure. That is, we relate the two occupation measures $\nu(I;-)$ and $\nu^*(I;-)$ for intervals $I\subset \R_+$ where
	   \begin{equation*}
	       \nu(I; A):= \Leb\{t\in I: X(t) \in A\},\qquad \nu^*(I; A):= \Leb\{t\in I: X^*(t)\in A\},
	   \end{equation*} and $\Leb$ denotes the Lebesgue measure. One of the main theorems we prove in Section \ref{sec:occupationChange} is the following:
   	\begin{thm} \label{thm:timeChange1}
		Fix a possibly random time $\tau\in [0,\infty]$ independent of $S$, and let $\tau^* = S(\tau)$ for some strictly increasing subordinator  $S$ (with the convention $S(\infty) = \infty$). Define the increasing processes $A$ and $A^*$ by
		\begin{equation*}
			A(v) = \nu([0,\tau]; (-\infty,v]),\qquad \text{and}\qquad A^*(v) = \nu^*([0,\tau^*];(-\infty,v]).
		\end{equation*} Then
		\begin{equation*}
			\left(A^*(v);v\in\R\right) \overset{d}{=} \left(S(A(v));v\in \R\right).
		\end{equation*}
	\end{thm}
	 \begin{rmk}
	 	The proof of Theorem \ref{thm:timeChange1} can be altered to prove an analogous result when the left endpoint of the intervals is not zero. 
	 \end{rmk}
	   
	 The Ray-Knight theorems give us detailed information on the form of the occupation measure for a standard Brownian motion \cite[Section XI.2]{RY.99}. These results have been extended in numerous ways. In particular, Warren \cite{Warren.97} describes the occupation measure of a Brownian motion with negative drift. We apply Theorem \ref{thm:timeChange1} to Warren's result and get Theorem \ref{thm:1d}. There is a similar Ray-Knight description for the occupation measure of a 3-dimensional Bessel process \cite{MY.08} which would allow one to prove a Bessel process version of Theorem \ref{thm:1d}. 
	 
	 \subsection{Overview of Paper}
	 
	 In Section \ref{sec:prelim} we discuss some preliminaries used frequently in the remainder of the paper. This includes a discussion on pricing Parisian options where the stock price is a geometric Brownian motion in Section \ref{sec:parisianPrice} and a separate discussion on the connection between time-fractional equations and their probabilistic representations of solutions in Section \ref{sec:timeFracPDE}.
	 
	 In Section \ref{sec:ExcursionTheorySection} we describe an excursion theory of time-changed Brownian motion. Since the time-change results in the process being non-Markov we do this using 
L\'{e}vy's construction of reflected Brownian motion. We recap the Brownian case in Section \ref{sec:ExcursionsForBM} and then include the time-change in Section \ref{sec:timeChangeExcursions} where we define the excursion measure. We then describe some basic properties of this measure in Section \ref{sec:basicn*}, wherein we prove Theorem \ref{thm:PD}. The remainder of the section is devoted to providing an explicit way to construct the excursion process of a time-changed Brownian motion from that of the original one.
	 
	 We then use this excursion theory to describe the first negative excursion of a time-changed Brownian motion $B^*$ which lasts longer than duration $D$. This is included in Section \ref{sec:firstExcursion}, where we find the Laplace transforms for the two measures $m_j$ of Theorem \ref{thm:priceTheorem}. In Section \ref{sec:illiquidMarkets}, we finish the proof of Theorem \ref{thm:priceTheorem}. Finally, in Section \ref{sec:occupationChange} we prove Theorem \ref{thm:timeChange1} and then use this idea to compute the Laplace transform of integral functionals of time-changed Brownian motion killed at zero in Section \ref{sec:rayKnight}.
	
\section{Preliminaries} \label{sec:prelim}

	\subsection{Poisson Random Measures}
	
	We briefly recall the definition of a Poisson random measure (PRM) and some of their basic properties. We do this in order to describe the PRM construction of subordinators. For more information see \cite[Chapters 1-2]{Kingman.93} and \cite[Chapter 8]{Cinlar.11}.
	
	A PRM with intensity $\mu$ on a measure space $(U,\mathcal{U},\mu)$ is a random discrete measure $N$ on $(U,\mathcal{U})$ such that for any finite collection of disjoint measurable sets $U_1,\cdots,U_n$, the random variables $N(U_1),\dotsm, N(U_n)$ are independent Poisson with means $\E[N(U_j)] = \mu(U_j)$. It is well-known that if $\mu$ is $\sigma$-finite
	then there exists a Poisson random measure with intensity $\mu$. 
	
	One of the first results about PRMs is the so-called mapping theorem which we include below. A proof of the result can be found in \cite[Chapter 2]{Kingman.93} in the case of atomless measures and \cite[Chapter 8]{Cinlar.11} in general.
	\begin{lem}\label{lem:mappingTheorem}
		Suppose that $N$ is a Poisson random measure on a space $(U,\mathcal{U})$ with intensity measure $\mu$. Given another measurable space $(V,\mathcal{V})$, and a measurable map $f:U\to V$, define a new measure on $V$ by $N^{\ast} = f_\#N = N\circ f^{-1}$. Then $N^{\ast}$ is a Poisson random measure on $(V,\mathcal{V})$ with intensity measure $f_\#\mu$. 
	\end{lem}
	
	\subsection{Subordinators}

	Increasing L\'{e}vy processes are called subordinators. Most of the results below are recalled from Chapter III of Bertoin's monograph \cite{Bertoin.96}. A subordinator $S = (S(t);t\ge 0)$ can be uniquely characterized by its Laplace transform, which takes the form
	\begin{equation}	\label{eqn:LaplaceExponent}
		\E[\exp(-\lambda S(t))] = \exp(-t\phi(\lambda)),\qquad \lambda,t\ge 0
	\end{equation} 
	where
	\begin{equation} \label{eqn:phiForLP}
		\phi(\lambda) = \kappa \lambda + \int\limits_{(0,\infty)} (1-e^{-\lambda r})\,\nu(dr),
	\end{equation} where $\kappa \ge 0$ and $\nu$ is a Radon measure on $(0,\infty)$ such that $\int _{(0,\infty)} (1\wedge r)\,\nu(dr)<\infty$. In the sequel we will say that $\phi$ is the Laplace exponent of $S$.
	
	The L\'{e}vy-It\^{o} decomposition of L\'{e}vy processes provides a simple description of subordinators in terms of PRMs. Namely, let $M$ be a PRM on $[0,\infty)\times(0,\infty)$ with intensity measure $\Leb\otimes \pi$, then the process $S$ defined by
	\begin{equation}\label{eqn:PRM_subordinator}
		S(t) = \kappa t + \int\limits_{[0,t]} \int\limits_{(0,\infty)} y\,M(ds,\,dy) = \kappa t + \int\limits_{(0,\infty)} y\,M([0,t],dy)
	\end{equation}
	is, as the notation suggests, a subordinator with Laplace exponent $\phi$ in \eqref{eqn:phiForLP}. See, for example, \cite[Chapter 1]{Bertoin.96} and \cite[Proposition 6.4.6]{Cinlar.11} for a proof. In the sequel we will implicitly use the above decomposition often without reference.

	\subsection{Parisian Options}\label{sec:parisianPrice}
	
	In this section we discuss how to price a Parisian option as discussed in \cite{CJY.97}, when the underlying stock price is governed by a geometric Brownian motion. That is, $X$ is the unique solution to
	\begin{equation}\label{eqn:gmb}
		dX(t) = X(t)\left( \mu \,dt + \sigma d W(t)\right),\qquad X(0) = x>0,
	\end{equation} where $W$ is a standard Brownian motion under the measure $\QQ$ (the reason for this choice of symbol will be apparent later). More concretely, the process $W$ is the canonical Brownian motion on the probability space $\Om = C([0,T]\to \R)$, and $\F_t$ is the standard augmentation of $\sigma(W(s); s\le t)$ under the Wiener measure $\QQ$. For simplicity, we assume that the interest rate $r$ is equal to $0$. This model is arbitrage-free and complete. 
	
	A European option with maturity date $T$ and payoff function $\Phi(X_T)$ is priced by 
	\begin{equation*}
 \E_\Q[ \Phi(X_T)],\qquad \text{where}\ \Q \text{ is the unique risk neutral measure}.
	\end{equation*}
	The measure $\Q$ is the unique measure which turns $X(t)$ into an $(\mathcal F_{t})$-martingale under $\Q$, and is given by Girsanov's theorem \cite[Chapter VIII]{RY.99} by
	\begin{equation*}
	\Q(A) = \int_A \exp\left(-\frac{\mu}{\sigma} W(T) - \frac{\mu^2}{2\sigma^2} T \right) \,d\QQ.
	\end{equation*}
	Under the measure $\Q$, the process $X$ satisfies
	\begin{equation}\label{eqn:XExplicit1}
	X(t) = x \exp\left(\sigma Z(t) - \frac{\sigma^2}{2} t \right),
	\end{equation} where $Z(t) = W(t) + \frac{\mu}{\sigma }t$ is a $\Q$-Brownian motion.

	Fix a barrier $L>0$, a duration $D>0$ and payoff function $\Phi$. The down-and-in European Parisian option pays out $\Phi(X(T))$ but is only valid on the event that the stock price $X$ stays strictly below level $L$ for duration at least $D$. Thus, the price of this down-and-in option, which we denote by $V_{d,i}(x,T,L,D,\Phi)$ is given by
	\begin{equation*}
	V_{d,i}(x,T,L,D,\Phi) = \E_\Q\left[1_{[H_{L,D}^-(X) \le T]}\Phi(X(T)) \right],
	\end{equation*}
	where $H_{L,D}^-$ is defined in \eqref{eqn:HLD_def}. 
	We can write $B(t) =Z(t) - \frac{\sigma}{2}t$, and use Girsanov's theorem to turn $B$ into a Brownian motion under a measure $\PR$ defined by
	\begin{equation*}
\frac{d \Q}{ d\PR} = \exp\left( \frac{\sigma}{2} B(T) - \frac{\sigma^2}{8} T\right) = \exp \left(\frac{\sigma}{2} Z(T) + \frac{\sigma^2}{8}T \right).
	\end{equation*}
	By \eqref{eqn:XExplicit1} $X(t) = x \exp(\sigma B(t))$. It is easy to see that an excursion of $X$ below level $L$ of duration $D$ is equivalent to an excursion of $B$ below level $b = \frac{1}{\sigma} \log (L/x)$ of duration $D$. Thus
	\begin{equation*}
	V_{d,i}(x,T,L,D,\Phi) = e^{-\sigma^2 T/8} \E_{\PR} \left[1_{[H_{b,D}^-(B) \le T]} \exp(\frac{\sigma}{2} B(T))\Phi(xe^{\sigma B(T)})\right].
	\end{equation*}
	The authors of \cite{CJY.97} give explicit formulas for the price of the Parisian call option, i.e. when $\Phi(x) = (x-K)_+$ for some price $K$.

	\subsection{Time-fractional partial differential equations for time-changed diffusions}\label{sec:timeFracPDE}
	
	
		Over the past $50$ years, fractional calculus has been drawn attention by a number of researchers. In addition to the theoretical aspect of its own \cite{SKM.1993}, \cite{K.2006}, its attractiveness is also due to the wide range of applicability in a number of areas including physics, chemistry, signal processing, modeling and control, economics and even social science \cite{MS.19, D.2003, P.1998, T.2019, S.2010}. To see more on the literature of the development of fractional calculus and its applications during 1974-2010, see \cite{M.2011}. 

In particular, fractional calculus has been widely used to model anomalous diffusion in physics. Anomalous diffusion is a diffusion whose individual particle travels at a different rate than in a traditional one. Especially, anomalous subdiffusion describes the phenomena of particle sticking or trapping in an inhomogeneous medium \cite{L.2010, MS.19}. The time-fractional heat equation 
\begin{equation}	\label{e:brownian}
	\partial_{t}^{\alpha} u(x,t) = \Delta u(x,t), 
\end{equation}
where $\partial^{\alpha}$ is the Caputo derivative with order $0<\alpha<1$ defined by 
\begin{equation}	\label{e:Caputo}
	\partial_{t}^{\alpha} f(t) = \frac{1}{\Gamma(1-\alpha)} \frac{d}{dt} \int_{0}^{t} (t-x)^{-\alpha} (f(x)-f(0)) \, dx,
\end{equation}
captures such subdiffusive behavior, just as the classical heat equation $\partial_{t} u(x,t) = \Delta u(x,t)$ describes normal diffusion. Here $\Gamma(x):= \int_{0}^{\infty} t^{x-1}e^{-x}dx$ is the Gamma function. In 2004, Meerschaert and Scheffler \cite[Theorem 5.1]{MS.04} proved, in light of \cite[Theorem 3.1]{BM.01}, that the solution of \eqref{e:brownian} with initial condition $u(x,0) = f(x)$ for some $f\in C^{2}$ such that $f$ and $\Delta f$ vanish at infinity can be represented probabilistically as
\begin{equation}	\label{e:probrep}
	u(x,t) = \mathbb E_{x}[ f(X(E(t)))]
\end{equation}
where $X$ is a scaled Brownian motion in $\mathbb R^d$ starting from $x$ whose infinitesimal generator is $\Delta$ and $E$ is the right-continuous inverse of an $\alpha$-stable subordinator $S$ independent of $X$ defined by \eqref{eqn:Edef}. 
The process $E$ is often called the inverse subordinator or the first hitting time process. In fact, the above-mentioned result still holds when the spatial process $X$ is a strong Markov process with a more general infinitesimal generator $\mathcal L$ in place of the Laplacian $\Delta$ in \eqref{e:brownian}. More recently, Chen \cite[Theorem 2.3]{Chen.17} proved that such interplay between time-fractional PDEs and the probabilistic representation of its solution works even when the time process $S$ is a subordinator satisfying more general conditions. The function $u(x,t)$ in \eqref{e:probrep} satisfies a more general time-fractional differential equation of the form
\begin{equation*}
	\partial_{t}^{S} u(x,t) = \mathcal L u(x,t) \quad \textnormal{with}\ u(x,0)=f(x)
\end{equation*}
where $X$ is a strong Markov process with infinitesimal generator $\mathcal L$ and $S$ is any increasing L\'{e}vy process with drift $\kappa$ and L\'{e}vy measure $\nu$ satisfying $\nu(0,\infty)=\infty$. Here $f$ belongs to the domain of the operator $\mathcal L$. 
The generalized time-fractional derivative above is 
\begin{equation}	\label{eqn:genCaputo}
	\partial_{t}^{S} f(t) = \kappa \frac{d }{d t} f(t) + \frac{d}{dt} \int_{0}^{t} w(t-x)(f(x)-f(0)) \, dx ,\qquad \text{where}\ \ w(x) = \nu(x,\infty).
\end{equation}
This reduces to (a constant multiple of) the classical $\alpha$-th order Caputo derivative \eqref{e:Caputo} with $0<\alpha<1$ when $S$ is an $\alpha$-stable subordinator. Moreover, Chen \cite{Chen.17} shows that this solution is unique in some strong sense. See therein for a detailed formulation. 


\section{The Excursion Theory of Time-changed Brownian motion}\label{sec:ExcursionTheorySection}
	
	Let us consider the following set-up. We have a standard linear Brownian motion $B = (B(t);t\ge 0)$ and we have an independent subordinator $S = (S(t);t\ge 0)$ which is assumed to be strictly increasing. We let $E = (E(t);t\ge 0)$ denote the right-continuous inverse of $S$ defined by \eqref{eqn:Edef}.

	By L\'{e}vy's construction of reflected Brownian motion \cite[Theorem VI.2.3]{RY.99} we know that $R = (R(t);t\ge 0)$ defined by
	\begin{equation*}
		R(t) = B(t) - I(t),\qquad I(t):= \inf_{s\le t} B(s)
	\end{equation*} is a reflected Brownian motion, and its local time at zero is $-I(t)$. We will relate excursions of the time-changed process $R^* = (R^*(t);t\ge 0)$ where $R^*(t) = R(E(t))$ to those of $R$. Similarly, we relate their excursion measures. 
	
	\subsection{The Excursions of a Reflected Brownian motion}\label{sec:ExcursionsForBM}
	Let us begin by recalling the construction of the It\^{o} measure for the reflected Brownian motion $R$. In order to do this we let $(\F_t;t \ge 0)$ denote the standard augmentation of the filtration $\sigma(B(u),S(u); u\le t)$. Now define
	\begin{equation*}
		\tau_x = \inf\{t: -I(t)>x\}.
	\end{equation*} The time $\tau_x$ is the first time when the Brownian motion $B$ is below level $-x$. The times $(\tau_x;x>0)$ index the excursion intervals of the reflected Brownian motion $R$.

	To define these excursions, we let $\EE$ be the set of positive excursions and a cemetery state:
	\begin{equation*}
		\EE = \left\{w\in C(\R_+,\R_+): w(0) = 0, \text{ and } \exists \zeta = \zeta(w)\text{ s.t. } w(t)>0 \text{ iff }t\in(0,\zeta) \right\} \cup\{\partial\}.
	\end{equation*} 
	Here $\partial$ is the cemetery state. This set is a Borel subset of $C(\R_+,\R_+)\cup\{\partial\}$ and when we refer to measurability in $\EE$ we mean with respect to the Borel $\sigma$-algebra as a subset of $C(\R_+,\R_+)$.

	We define the excursions process $(e_x; x>0)$ of $R$ by
	\begin{equation*}
		e_x(t) = \left\{\begin{array}{ll}
		R((t+\tau_{x-})\wedge \tau_x)&: \tau_x- \tau_{x-}>0\\
		\partial &: \text{else}
		\end{array}
		\right..
	\end{equation*}
	It is well known that $(e_x;x>0)$ is a Poisson point process on $\EE$ with respect to the filtration $(\F_{\tau_x};x>0)$. Its intensity measure is the It\^{o} measure (of positive excursions) on $\EE$, which is defined as
	\begin{equation*}
		\n(A):= \frac{1}{x} \E\left[\sum_{0<y\le x} 1_{[e_y\in A]}\right].
	\end{equation*} The scalings of this measure vary throughout the literature depending on conventions of the local time, so we will mention the following identities useful in determining the scaling:
	\begin{equation}\label{eqn:easyPropsforn}
		\n(\zeta>\eps) = \sqrt{\frac{2}{\pi \eps}} ,\qquad\n\left( \sup_{t\le \zeta} w(t)> \eps\right) = \frac{1}{\eps},\qquad \forall \eps>0.
	\end{equation} See \cite{LeGall.10} for a proof with a different scaling; however, see around \eqref{eqn:laplace(t)au} for a proof that \eqref{eqn:easyPropsforn} is the correct scaling.

	\subsection{The time-change}\label{sec:timeChangeExcursions}
	
	In this section we discuss in more detail the time-change $t\mapsto E(t)$ and its relation to the excursions of $R^*$. We let $B^*(t) = B(E(t))$ and $I^*(t) = \inf_{s\le t} B^*(s) = I(E(t))$. Let us write $\F^*_t$ for the usual augmentation of $\sigma(B(E(s)),s\le t)$. 
	
	In order to index the excursions of $R^*$, we need some analog for $\tau_x$. Since $R^*$ is not Markovian, it is difficult to construct a useful local time at zero from just $R^*$. However, by Skorohod reflection, $-I^*(t)$ acts as a local time at level zero in the sense that $-I^*(t)$ only increases on the set $\{t:R^*(t) = 0\}$. Therefore we set $\hat{\tau}_x$ as
	\begin{equation*}
		\hat{\tau}_x = \inf\{t: -I^*(t)>x\}.
	\end{equation*} We observe the following simple relationship between the times $\hat{\tau}_x$ and $\tau_x$ which will be used numerous times in the sequel:
	\begin{equation}\label{eqn:hatTauandS}
		\hat{\tau}_x = \inf\{t:-I(E(t))>x\} = \inf\{S(u): -I(u)>x\} = S({\tau_x}).
	\end{equation} Indeed, the second equality occurs with probability 1. The process $S$ almost surely does not jump on any independent Lebesgue null set. To be more specific, it will not jump on the independent null set $\{t: R(t) = 0\}$. The third equality holds by the right-continuity of the subordinator $S$.
	
	The times $\hat{\tau}_x$ are easily seen to be $\F^*_t$-stopping times. Moreover, $\F^*_{\hat{\tau}_x} = \F_{\tau_x}$. Indeed, since $E = \langle B(E(\cdot))\rangle$ is the quadratic covariation of $B(E(t))$, and so the filtration $\F^*_t$ is the standard augmentation of $\sigma (B(E(s)),E(s); s\le t)$. Since $S(t) = \inf\{u: E(u)>t\}$, $S(t)$ is an $\F^*_t$-stopping time. The claim that $\F^*_{\hat{\tau}_x} = \F_{\tau_x}$ follows from $\F^*_{S(t)} = \F_t$ and \eqref{eqn:hatTauandS}. We now argue that $\F^*_{S(t)} = \F_t$.
	
	First, $\F_{t}\subset \F_{S(t)}^*$. Indeed, for any fixed $r\le t$, $B(r) = B(E(S(r)))$ and $S(r)$ are both $\F_{S(t)}^*$-measurable. Hence, $\sigma(B(r),S(r); r\le t)\subset \F_{S(t)}^*$ and so $\F_t\subset \F_{S(t)}^*$. 
	
	To show the converse, we begin by showing that $\F^*_t \subset \F_{E(t)}$. The latter $\sigma$-algebra is well defined because $E(t)$ is an $(\F_t)$-stopping time for each fixed $t$. Now, for each $r\le t$, $B(E(r))$ is clearly $\F_{E(t)}$-measurable. So $\sigma(B(E(r));r\le t)\subset \F_{E(t)}$ and hence $\F_t^*\subset \F_{E(t)}$. Now as previously noted, $S(t)$ is a stopping time with respect to both filtrations $(\F_t^*)$ and $(\F_{E(t)})$. Hence, $\F^*_{S(t)}\subset \F_{E(S(t))}$. However, since $S$ is strictly increasing, $E(S(t)) = t$ almost surely, and since null sets are contained in $\F_0$ we have $\F^*_{S(t)}\subset \F_{E(S(t))} = \F_t$. Hence $\F_{S(t)}^* = \F_t$ as desired.
	
	

	We now define the following excursion process $(e^*_x;x>0)$ of $R^*$ defined by
	\begin{equation*}
		e^*_x(t) := \left\{ \begin{array}{ll}
			R^*((t+\hat{\tau}_{x-})\wedge \hat{\tau}_x)&: \hat{\tau}_{x}-\hat{\tau}_{x-}>0,\\
			\partial &: \text{else}
			\end{array}
			\right..
	\end{equation*}
 	Since the process $S$ does not jump when $R(t) = 0$, we observe that $e_x^*$ almost surely takes values in the set $\EE$. That is, there is no initial trapping at zero.
 	
 	An important ingredient in the proof that $(e_x;x>0)$ is a Poisson point process is that the process $R$ is a strong Markov process. Unfortunately, the process $R^*$ is not, in general, a Markov process: intervals of constancy for the process $R^*$ can tell us information about the future of the process. In order to overcome this problem, we prove the following proposition.
 	
 	\begin{prop}\label{prop:increments}
 	Fix an $x>0$. Define
 	\begin{equation*}
 		\tilde{R}(t):= R^*(\hat{\tau}_x+t),\qquad t\ge 0.
 	\end{equation*} Then $\tilde{R}\overset{d}{=} R^*$ and is independent of $\F^*_{\hat{\tau}_x} = \F_{\tau_x}$.
 	\end{prop}
	\begin{proof}
		We begin by noting from \eqref{eqn:hatTauandS} that $\hat{\tau}_x + t = S(\tau_x)+t$. From here observe
		\begin{equation*}
			E(S(\tau_x)+t) = \inf\{u: S(u)>S(\tau_x)+t\} = \inf\{\tau_x+v: S(\tau_x+v)- S(\tau_x)>t\} = \tau_x+ \tilde{E}(t)
		\end{equation*} where $\tilde{E}$ is the right-continuous inverse of $\tilde{S} = (\tilde{S}(v):= S(\tau_x+v)-S(\tau_x);v\ge 0)$. Since $S$ is a L\'{e}vy process it has stationary increments and since $\tau_x$ is an independent random time (it only depends on the Brownian path $B$), the processes $\tilde{S}$ and $S$ have the same law. Hence $\tilde{E}$ and $E$ have the same law. Moreover, both $\tilde{S}$ and $\tilde{E}$ are independent of $\F_{\tau_x}$ since $S$ is strong Markov.
	
		Hence 
		\begin{equation*}
			\tilde{R}(t) = R^*(\hat{\tau}_x+t) = R(E(S(\tau_x)+t)) = R(\tau_x+\tilde{E}(t)).
		\end{equation*} By the strong Markov property for $R$, the right-most term, as a process in $t$, is easily seen to be equal in distribution to $R^*$ and is independent of $\F_{\tau_x}$. This proves the desired claim.
	\end{proof}

	With the above proposition, we can now prove that $(e^*_x;x>0)$ is an $(\F_{\tau_x})_{x} = (\F^*_{\hat{\tau}_x})_x$-Poisson point process.
	\begin{lem}
	$(e^*_x;x>0)$ is an $(\F^*_{\hat{\tau}_x})$-Poisson point process.
	\end{lem}
	
	\begin{proof}
		We first begin by observing that it is $\sigma$-discrete. Indeed, the time-change does not change the maximum of an excursion, so there are only a finite number excursions of $(R^*(t); 0\le t\le \hat{\tau}_x)$ which exceed level $\eps$ for any $\eps>0$. It is also easy to verify that $(x,\w)\mapsto e_x^*(\w)$ is a measurable with respect to $\mathcal{B}\otimes \F_\infty$, where $\mathcal{B}$ is the Borel $\sigma$-algebra on $\R$.
		
		We must now verify the independence of the increments. We do this by using Proposition \ref{prop:increments}. Towards this end, fix a set measurable set $\Gamma\subset \EE$ and define
		\begin{equation}\label{eqn:N_gammaDef}
			N_\Gamma^* (a,b] = \sum_{a<y\le b} 1_{[e_y^*\in \Gamma]}.
		\end{equation}
		We must show that $N_\Gamma^*(x):= N_{\Gamma}^*(0,x]$ is an $\F_{\tau_\cdot }$-Poisson process. However, by Proposition \ref{prop:increments} it is easy to see
		\begin{equation*}
			\left(N_\Gamma^*(y,y+x]\Big| \F_{\tau_y}\right) \overset{d}{=} N_\Gamma^*(0,x],\qquad \forall x,y>0. 
		\end{equation*} Indeed, the local time of the process $\tilde{R}(t) := R^*(\hat{\tau}_y+t)$ is simply $-\tilde{I}(t)=-I^*(t)-y$. Let $\tilde{\tau}$ denote the right-continuous inverse of $-\tilde{I}$. So the excursions of $\tilde{R}$ occurring before time $\tilde{\tau}_x$ are simply those of $R$ occurring between times $\hat{\tau}_y$ and $\hat{\tau}_{y+x}$. This proves the desired claim.
		
		\end{proof}

		We now define the excursion measure $\n^*$ on $\EE$ by
		\begin{equation*}
			\n^*(\Gamma) = \frac{1}{x} \E\left[N_{\Gamma}^*(x)\right]
		\end{equation*} where $N_\Gamma^*$ is defined in \eqref{eqn:N_gammaDef}.
	
	\subsection{Basic Properties of $\n^*$}\label{sec:basicn*}
	
	We now wish to describe some path properties of functions $w\in \EE$ under the law $\n^*$. We begin by describing the analogues of properties in \eqref{eqn:easyPropsforn} for the measure $\n$.
	
	The supremum is a much easier object to study so we begin with that.
	\begin{lem}
		For each $\eps>0$, $\n^*(\sup_{t\le \zeta} w(t)>\eps) = \frac{1}{\eps}$. 
	\end{lem}
	\begin{proof}
		For It\^{o}'s measure on positive excursions, we know that $\n(\sup w(t)>\eps) = \frac{1}{\eps}$. 
		We have
		\begin{equation*}\sup_{t}e_x(t) =	\sup_{\tau_{x-} \le t\le \tau_x} (B(t)-I(t)) = \sup_{{\hat{\tau}_{x-}}\le t\le \hat{\tau}_x} (B^*(t)-I^*(t)) = \sup_{t}e_x^*(t).		
		\end{equation*} Hence
	\begin{equation*}
		\n^*(\sup_t w(t)>\eps) = \E\left[ \sum_{0<x\le 1} 1_{[\sup_t e_x^*(t)>\eps]} \right] = \E\left[ \sum_{0<x\le 1} 1_{[\sup_t e_x(t)>\eps]} \right] = \n(\sup_t w(t)>\eps),
	\end{equation*} as desired.
	\end{proof}

	Determining the distribution of the length of the excursions is a little trickier. To begin let us recall a proof of the fact $\n(\zeta>\eps) = \sqrt{\frac{2}{\pi \eps}}$. For each $\eps>0$, set
		\begin{equation*}
		N_\eps(x) = \sum_{0<y\le x} 1_{[\zeta(e_y)>\eps]} = \sum_{0<y\le x} 1_{[\tau_y-\tau_{y-}>\eps]} = \#\{y\in (0,x]: \Delta \tau_{y} >\eps\}.
		\end{equation*} We recall that $\tau_x$ is a 1/2-stable subordinator and we can write its Laplace transform \cite[Theorem 6.2.1]{KS.91} as
		\begin{equation} \label{eqn:laplace(t)au}
		\E[\exp(-\lambda \tau_x)] = \exp(-x\sqrt{2\lambda}) = \exp\left( -x   \int\limits_{(0,\infty)} (e^{-\lambda r} -1 )\,\nu(dr)\right),
		\end{equation} where $\nu$ is the L\'{e}vy measure of $\tau$. Let $M$ be the Poisson random measure on $[0,\infty)\times (0,\infty)$ with intensity $\Leb\otimes \nu$. Then
		\begin{equation*}
		\left(\tau_x; x\ge 0\right) \overset{d}{=} \left(\int\limits_{y= 0}^x \,\int\limits_{z= 0}^\infty z \, M(dy,dz); x\ge 0 \right).
		\end{equation*}
		So, in particular,
		\begin{equation*}
		N_{\eps}(x) = \#\{y\in(0,x]: \Delta\tau_y >\eps\} \overset{d}= M((0,x]\times(\eps,\infty))
		\end{equation*} and hence
		\begin{equation*}
		\n(\zeta>\eps) = \frac{1}{x}\E[N_\eps(x)] = \frac{1}{x} \E[M((0,x]\times (\eps,\infty))] = \nu(\eps,\infty).
		\end{equation*} 
		
		To calculate $\nu(\eps,\infty)$, apply integration by parts to the integral in the right-most term of \eqref{eqn:laplace(t)au} to get
		\begin{equation}\label{eqn:IBPforLevy}
		\sqrt{2\lambda} = \lambda \int_0^\infty e^{-\lambda r}\nu(r,\infty)\,dr.
		\end{equation} Said another way, the function $r\mapsto \nu(r,\infty)$ has Laplace transform $\sqrt{2\lambda}/\lambda = \displaystyle \sqrt{\frac{2}{\lambda}}$. Inverting this Laplace transform gives
		\begin{equation*}
		\nu(\eps,\infty) = \sqrt{\frac{2}{\pi \eps}}.
		\end{equation*}

	One can alter the approach above to find the value of the measure of $\n^*(\zeta>\eps)$. Indeed, note that by \eqref{eqn:hatTauandS} we have
	\begin{equation*}
	\begin{split}
	N_\eps^*(x)&:= \sum_{0<y\le x} 1_{[\zeta(e_y^*)>\eps]} = \#\{y\in(0,x]: \Delta \hat{\tau}_y>\eps\}\\
	&= \#\{y\in(0,x]: S(\tau_y)-S(\tau_{y-})>\eps \}.
	\end{split}
	\end{equation*}
	We now make the following claim. 
	\begin{claim}
	If $S = (S(t);t\ge 0)$ is a subordinator with Laplace exponent $\phi$ and $\tau = (\tau_x;x\ge 0)$ is as above, then $(S(\tau_x);x\ge 0)$ is a subordinator with Laplace transform
	\begin{equation*}
	\E[\exp(-\lambda S(\tau_x))] = \exp(-x \sqrt{2\phi(\lambda)}).
	\end{equation*}
	\end{claim}
The fact that the process $S(\tau_{\cdot})$ is a subordinator without drift follows from the fact that both $S$ and $\tau$ are L\'{e}vy, see \cite[Theorem 30.1]{Sato99} for instance, and that $\tau$ increases solely by jumps. The computation of the Laplace transform of $S(\tau_{\cdot})$ is a simple conditioning argument. In general, one should not expect a simple representation of the L\'{e}vy measure of such subordinators $S(\tau_\cdot)$. See a representation involving Fourier transforms and fractional convolutions in \cite[Theorem 30.1]{Sato99}. However, it is possible to spell it out explicitly in some cases. See Corollary \ref{cor:betaover2} below. 
	
	We now state the following lemma.
	\begin{lem}\label{lem:LevyMeasureForTCIto}
		Let $\nu^*(dr)$ denote the L\'{e}vy measure of $(S(\tau_x);x\ge 0)$. Then, for each $\eps>0$, we have
		\begin{equation*}
		\n^*(\zeta>\eps) = \nu^*((\eps,\infty)).
		\end{equation*}
	\end{lem} 
	\begin{proof}
		The proof of this claim is simply a repetition of the argument that $\n(\zeta>\eps) = \nu((\eps,\infty))$ with the appropriate replacements of $N_\eps$ with $N_\eps^*$ and the like. 
	\end{proof}

	\subsubsection{Properties of $\n^*$ when $S$ is $\beta$-stable}

	Let us also focus shortly for the case where $S$ is a $\beta$-stable subordinator for some $\beta\in(0,1)$, i.e. its Laplace exponent is $\phi(\lambda) = c \lambda^\beta$. Then $S(\tau_x)$ is a $\frac{\beta}{2}$-stable subordinator with Laplace exponent 
	\begin{equation*}
		\E\left[ \exp(-\lambda S(\tau_x))\right] = \exp(-x \sqrt{2c} \lambda^{\beta/2}).
	\end{equation*} In particular its L\'{e}vy measure $\nu^*$ can be calculated as in \eqref{eqn:IBPforLevy} by inverting its Laplace transform:
	\begin{equation*}
	    \sqrt{2c} \lambda^{\beta/2} = \lambda \int_0^\infty e^{-\lambda r} \nu^*((r,\infty))\,dr. 
	\end{equation*}
	
		Lemma \ref{lem:LevyMeasureForTCIto} implies the following corollary, which is obtained by just inverting Laplace transform above.
	\begin{cor} \label{cor:betaover2}
		Let $S$ be a $\beta$-stable subordinator with Laplace exponent $\phi(\lambda) = c\lambda^\beta$ for some $\beta\in(0,1)$ and constant $c>0$. Then 
		\begin{equation*}
		\n^*(\zeta>\eps) = \frac{\sqrt{2c}}{\Gamma \left(1-\frac{\beta}{2}\right)} \eps^{-\beta/2}.
		\end{equation*}
	\end{cor}

	We note that the jumps of the subordinator $S(\tau_x)$ are fairly well studied \cite{PPY.92,PY.92, PY.97} as mentioned in the introduction, with a slightly different formulation. Let $Z\subset \R_+$ denote the closure of the range of the subordinator $S(\tau_x)$. The set $Z$ is of measure zero. For some (possibly random) value $T$, denote by $V_1(T)\ge V_2(T)\ge \dotsm$ the ordered lengths of the intervals of $[0,T]\setminus Z$. The following proposition can be found in \cite{PY.92}.
	\begin{prop} \label{prop:PY92}
	For any $t>0$:
	\begin{equation*}
		\left(\frac{V_1(S(\tau(t)))}{S(\tau(t))}, \frac{V_2(S(\tau(t)))}{S(\tau(t))}, \dotsm \right) \overset{d}{=}\left( \frac{V_1(t)}{t},\frac{V_2(t)}{t},\dotsm\right) \overset{d}{=} \PD\left(\frac{\beta}{2}, 0\right),
	\end{equation*} where $\PD(\alpha,\theta)$ is the Poisson-Dirichlet distribution with parameters $\alpha$ and $\theta$.
	\end{prop}

    We now prove Theorem \ref{thm:PD}. 
    \begin{proof}[Proof of Theorem \ref{thm:PD}]
    
    We will maintain the notation used in this section, and let $R$ denote the reflected Brownian motion and $R^{*}$ its time-change by the inverse of a $\beta$-stable subordinator.


The excursion intervals of the process $R^*$ on the interval $[0,t]$ (including the almost-surely unfinished excursion straddling time $t$) are described by
	\begin{equation*}
	\left\{	\left(S(\tau_{x-}), S(\tau_x)\right) : S(\tau_x)\le t\right\} \cup \left\{\left(g^*(t),t\right] \right\},
\quad  \text{where}\quad g*(t) = \sup\{u\le t: R^*(u) = 0\}. 
	\end{equation*}
	Note that the zero set $\{u: R^*(u) = 0\} = \{S(\tau_{x-}), S(\tau_{x}): x\ge 0\}$, which is precisely the closure of the range of $S(\tau_\cdot)$. Therefore, 
	\begin{equation*}
		\{u\in[0,t]: R^*(u)>0\} = [0,t]\setminus Z,
	\end{equation*} 
	where $Z$ is the closure of the range of the $\beta/2$-stable subordinator $S(\tau_\cdot)$. Proposition \ref{prop:PY92} gives the desired result.

    \end{proof}

    Explicit formulas for the distribution of $V_1(t)$ the longest excursion interval can be found in \cite{PY.97}, and are mainly located in Section 3.2 therein. While these are rather clunky, the following corollary can be obtained from Corollaries 3 and 12 in \cite{PY.97}.

	\begin{cor} \label{cor:mainCor} Let $R^*$ be a reflected Brownian motion time-changed by the inverse of stable subordinator with Laplace transform \eqref{eqn:beta_Stable}.
		\begin{enumerate}
			\item Fix $t>\eps>0$. Let $J_\eps(t)$ count the number of excursions of $R^*$ that occur before time $t$ (including the last unfinished excursion) that which are length at least $\eps$. Only count that last excursion if $t-g^*(t)>\eps$. Then
			\begin{equation*}
				\begin{split}
				\E\left[\#J_\eps(t)\right]
				&= \frac{1}{\Gamma(\beta/2)\Gamma(1-\beta/2)}\int_{\eps/t}^1 \frac{(1-u)^{\frac{\beta}{2}-1}}{u^{\frac{\beta}{2}+1}} \: du\\
				&=  \frac{1}{\Gamma(\beta/2)\Gamma(1-\beta/2)} \frac{2}{\beta} \left(\frac{t}{\eps} - 1\right)^{\beta/2}.
				\end{split}
			\end{equation*}
		\item Let $V_1(t)$ be the length of longest (possibly unfinished) excursion interval of $R^*$ on $[0,t]$ then
		\begin{equation*}
			\E\left[ \exp\left( - \lambda t /V_1(t)\right)\right] = \frac{e^{-\lambda} }{\displaystyle \Gamma(1-\beta/2) \lambda^{\beta/2} + \frac{\beta}{2} \int_1^\infty e^{-\lambda r}  r^{-1-\beta/2}\,dr}.
		\end{equation*}
		\end{enumerate}
	\end{cor}

	\subsection{Another representation of $\n^*$}
	
	In this section we present a different construction of the measure $\n^*$ which is useful for coupling the excursions of a Brownian motion $B$ and its time-change $B^*$.
	
 	Let $\M$ denote the space of Radon measures on $[0,\infty)\times (0,\infty)$ and let $\M_c$ denote the subset of measures $\mu$ on $[0,\infty)\times (0,\infty)$ such that 
 	\begin{equation*}
 	\mu(A) = \sum_{j=1}^\infty 1_{[(t_j,x_j)\in A]},
 	\end{equation*}
 	where $t_j$ are distinct elements in $(0,\infty)$ and $x_j\in (0,\infty)$.
 	
 	We will now consider the functions $\Xi:\EE\times \M_c\times [0,\infty) \to \D(\R_+,\R)$ by
 	\begin{equation*}
 	\Xi(w,\mu, \kappa)(t) = w\left( \inf\left\{u: \kappa u+ \int_0^\infty y \mu([0,u]\times dy) >t \right\}\right).
 	\end{equation*} We implicitly assume that either $u\mapsto \int_0^\infty y\mu([0,u],dy)$ is strictly increasing or $\kappa>0$ so that $\Xi(w,\mu,\kappa)\in \EE$.
 	
 	We now prove
 	\begin{thm} \label{thm:excursionChange}
	Suppose that the strictly increasing subordinator $S$ has L\'{e}vy-It\^{o} decomposition \eqref{eqn:PRM_subordinator} where $M$ has intensity measure $\Leb\otimes \nu$. Then
 	\begin{equation*}
 	\n^* = \Xi(\cdot,\cdot, \kappa)_\# (\n\otimes \mathcal{L}(M))
 	\end{equation*}
 	where $\mathcal{L}(M)$ is the law of the random measure $M$.
 	\end{thm}
 	\begin{proof}
 		We first observe that $\n^*(\sup_t w(t)>\eps) = \n(\sup_t w(t)) = \frac{1}{\eps}$ and that the map $\Xi$ does not change the supremum of the of the process. Hence, it suffices to show for each $\eps>0$
 		\begin{equation*}
 		\n^*\left(\ \bullet \ \bigg| \sup_t w(t)>\eps\right) = \Xi(\cdot, \cdot, \kappa)_{\#} \left(\n\left(\ \bullet \ \bigg| \sup_t w(t)>\eps\right)\otimes \mathcal{L}(M)\right).
 		\end{equation*}
 		
 		To prove this we let $e^*(t)$ be the first excursion of $R^*$ which exceeds level $\eps$, and $e(t)$ be the first excursion of $R$ which exceeds level $\eps$. Formally, set $T^* = \inf\{t: R^*(t) > \eps\}$, $g^* = \sup\{r<T^*: R^*(r) = 0\}$ and $d^* = \inf\{r>T^*: R^*(r) = 0\}$ then
 		\begin{equation*}
 		e^*(t) = R^*((g^*+t)\wedge d^*).
 		\end{equation*} Similarly define $T$, $g$ an $d$ in the construction of $e$. Observe that $g^* = S(g)$ by \eqref{eqn:hatTauandS}. By \cite[Proposition XII.1.13]{RY.99} the first hitting time of a set $A$ of a Poisson point process with intensity $\mu$ has distribution $\mu(-|A)$, provided that $\mu(A)<\infty$. Hence both $e^*\sim \n^*(-|\sup_t w(t)>\eps)$ and $e\sim \n(-|\sup_t w(t)>\eps)$. Therefore, it suffices to prove
 		\begin{equation*}
 		e^* \overset{d}{=} \Xi(e,\tilde{M},\kappa)
 		\end{equation*} where $\tilde{M}\overset{d}{=} M$ is independent of $e$.
 		
 	 	Define $\tilde{M}([s,t]\times (a,b)) = M([g + s, g+t]\times (a,b))$. Note that $g$ is independent of $M$ ($g$ depends only on the reflected Brownian motion $R$), and so by conditioning on $g$ it is easy to see that $\tilde{M}$ is a Poisson random measure on $[0,\infty)\times(0,\infty)$ with intensity $\Leb\otimes \nu$. 
 	 	
 	 	Next, for $t\le d^*-g^*$, and as in the proof of Proposition \ref{prop:increments},
 	 	\begin{equation*}
 	 	e^*(t) = R( E( S(g)+ t)) = R(E(S(g)) + \tilde{E}(t))= R (g+\tilde{E}(t)) = e(\tilde{E}(t))
 	 	\end{equation*} where
 	 	$\tilde{E}(t) = \inf\{u: S(u+g)-S(g)>t\}.
 	 	$ Since $g$ is independent of $S$ and $S$ has stationary increments, $\tilde{E}\overset{d}{=} E$.
 	 	
 	 	But, by the L\'{e}vy-It\^{o} decomposition of $S$, we have
 	 	\begin{equation*}
 	 	S(u+g) - S(g) = \kappa u  + \int_0^\infty y \,M([g,g+u],dy) = \kappa u + \int_0^\infty y\tilde{M}([0,u],dy).
 	 	\end{equation*} Hence
 	 	\begin{equation*}
			e^*(t) = e(\tilde{E}(t)) = \Xi\left( e, \tilde{M}, \kappa\right)
 	 	\end{equation*}
 	 	proving the desired claim.
 		
 	\end{proof}
 

	\subsection{Excursion Construction of Time-changed Brownian motion}
	
	We now recall that from the excursion process $e_x^*$ we can reconstruct the time-changed reflected Brownian motion $R^*$ by
	\begin{equation*}
	\begin{split}
	\hat{\tau}_x = \sum_{s\le x} \zeta(e_s^*),\qquad
	-I^*(t) = \inf\{x: \hat{\tau}_x>t\},\qquad 
	R^*(t) = \sum_{s\le -I^*(t)} e_s(t-\hat{\tau}_{s-})
	\end{split}.
	\end{equation*}
	The proof of this result is exactly the same as the Brownian case, which can be found in \cite[Proposition XII.V.2]{RY.99}. 
	
	We also recall that we can construct a standard Brownian motion from the It\^{o} measure $\n$ of \textit{positive} excursions as follows. Let $(e_x,\gamma_x)$ be a Poisson point process taking values in $\EE \times \{\pm 1\}$ with intensity 
	\begin{equation*}
	\n\otimes \left( \frac{1}{2} \delta_{-1} + \frac{1}{2} \delta_{+1}\right).
	\end{equation*} Then the process $B$ constructed by
	\begin{equation}\label{eqn:excursionBM}
	\begin{split}
	\tau_x = \sum_{s\le x} \zeta(e_s),\qquad  
	L(t) = \inf\{x: \tau_x>t\},\qquad 
	B(t) &= \sum_{s\le L(t)} \gamma_s e_s(t-\tau_{s-})
	\end{split}
	\end{equation} is a standard Brownian motion.
	
	But now let $S$ be a strictly increasing subordinator with L\'{e}vy-It\^{o} decomposition \eqref{eqn:PRM_subordinator}. Suppose this is independent of the Poisson point process $(e_x,\gamma_x;x>0)$ and let $E$ be its inverse defined in \eqref{eqn:Edef}. Note that, with probability 1, 
	\begin{equation*}
	M\left(\{\tau_{x-},\tau_x:x\ge 0\} \times (0,\infty)\right) = 0,
	\end{equation*} since $\{\tau_{x-},\tau_x: x\ge 0\}$ is the zero set of the Brownian motion $B$ and has Lebesgue measure zero.
	
	We now work conditionally on $(\tau_x;x\ge 0)$. Now define $\hat{\tau}_x = S(\tau_x)$. Then, for $\tau_{x-}\le E(t)\le \tau_x$, the value of
	\begin{equation*}
	\begin{split}
	E(t) &= \inf\{u: S(u)>t\} = \inf\{u+ \tau_{x-}: S(u+\tau_{x-}) - S(\tau_{x-}) > t-S(\tau_{x-})\} \\
	&= \tau_{x-}+ \inf\{u: S_{x}(u) > t-\hat{\tau}_{x-}\},
	\end{split}
	\end{equation*} 
	where $S_x(u) = S(u+ \tau_{x-})-S(\tau_{x-}) = \kappa u + \int_0^\infty y M([\tau_{x-},u],dy)$. Define, for $a\le b\le \tau_x-\tau_{x-}$
	\begin{equation*}
	M_x([a,b]\times A) = M([\tau_{x-}+a, \tau_{x-}+b]\times A).
	\end{equation*} Conditionally given $\tau$, the measure $M_x$ is a Poisson random measure on $[0,\tau_{x}-\tau_{x-}]\times (0,\infty)$ with intensity $\Leb\otimes \nu$. Moreover, the measures $\{M_x; x\ge 0\}$ are conditionally independent given $(\tau_x)_{x\ge 0}$ because $S$ has independent increments.
	
	But now by Theorem \ref{thm:excursionChange}, $(e_x^*, \gamma_x; x>0)$ where $e_x^*$ is defined by 
	$e_x^* = \Xi\left(e_x,M_x,\kappa\right)$ 
	is a Poisson point process on $\EE\times \{\pm1\}$ with intensity measure $\n^*\otimes\left(\frac{1}{2} \delta_{-1} + \frac{1}{2} \delta_{1} \right)$. Moreover it is not hard to see $e_x^*(t-\hat{\tau}_{x-}) = e_x(E(t) - \tau_{x-}).$
	Indeed, since $\hat{\tau}_{x-} = S(\tau_{x-})$
	\begin{equation*}
	\begin{split}
	\inf&\left\{u: \kappa u + \int_0^\infty y M_{x}([0,u],dy) > t - \hat{\tau}_{x-}\right\}= \inf\left\{u: \kappa (u+\tau_{x-}) + \int_0^\infty y M([0,\tau_{x-}+u],dy)>t \right\}\\
	&= \inf\left\{u: \kappa u + \int_0^\infty y M([0,u],dy) >t \right\} - \tau_{x-}.
	\end{split}
	\end{equation*}
	Also $\zeta(e_x^*) = S_x (\zeta(e_x))$.
	
	Hence
	\begin{equation*}
	B(E(t)) = \sum_{s\le L(E(t))} \gamma_s e_s \left(E(t) - \tau_{s-} \right) = \sum_{s\le L(E(t))} \gamma_s e_s^*(t-\hat{\tau}_{x-}),
	\end{equation*} and
	\begin{equation*}
	L(E(t)) = \inf\{x: \tau_x> E(t)\} = \inf\{x: \hat{\tau}_{x}>t\},
	\end{equation*} where we used the fact that, with probability 1, $S$ does not jump at any of the times $\{\tau_{x-},\tau_x; x\ge 0\}$.
	Hence we have proved the following theorem.
	\begin{thm} \label{thm:pppfortc}
		Let $(e_x^*,\gamma_x)$ be a Poisson point process on $\EE\times \{\pm1\}$ with intensity $\n^*\otimes (\frac{1}{2} \delta_{-1}+\frac{1}{2} \delta_1)$. Then $B^*$ defined by
		\begin{equation} \label{eqn:excurionABM}
		\hat{\tau}_x = \sum_{s\le x} \zeta(e_x^*),\qquad L^*(t) = \inf\{x: \hat{\tau}_x>t\},\qquad B^*(t) = \sum_{s\le L^*(t)} \gamma_s e_s^*(t-\hat{\tau}_{s-})
		\end{equation} is a Brownian motion time-changed by the inverse of an independent strictly increasing subordinator $S$. That is $B^* \overset{d}{=} (B(E(t));t\ge 0)$ for a Brownian motion $B$ and an independent inverse subordinator $E$.
	\end{thm}

	We observe that $E(t) = \langle B^* \rangle(t)$ is the quadratic covariation of $B^*$. Hence Theorem \ref{thm:pppfortc} allows us to construct the time-change $E$ and, subsequently, the process $S$. In fact, we can use Theorems \ref{thm:excursionChange} and \ref{thm:pppfortc} to obtain the following corollary:
	\begin{cor} \label{cor:EFromExcursion} Let $S$ be a strictly increasing subordinator with L\'{e}vy-It\^{o} decomposition \eqref{eqn:PRM_subordinator}.
	Let $(e_x, M_x, \gamma_x)_{x>0}$ be Poisson point process on $\EE\times \M_c \times \{-1,1\}$ with intensity measure $\n\otimes \mathcal{L}(M)\otimes (\frac{1}{2}\delta_{-1}+ \frac{1}{2}\delta_1)$. If $e_x^* = \Xi(e_x,M_x,\kappa)$ then $B^*$ defined in \eqref{eqn:excurionABM} is the time-change of the Brownian motion defined in \eqref{eqn:excursionBM} with time-change $E$ satisfying
	\begin{equation*}
	E(t) = \sum_{s\le L^*(t) } \langle e_x^*\rangle (t-\hat{\tau}_{s-}).
	\end{equation*}
	
	\end{cor}
	\subsection{Some marginal distributions for $\n$ and $\n^*$}
	
We now discuss an important relationship between $\n$ and $\n^*$ which will be useful in the sequel. Before covering this relationship we remark that both $\n$ and $\n^*$ are $\sigma$-finite measures; however, we can talk about processes have $\n$ (or $\n^*$) as its ``law'' by restricting to sets of finite $\n$ (resp. $\n^*$) mass.
 
We describe the ``law'' of $(\zeta(w^*), \langle w^*\rangle (\zeta(w^*)))$ under $\n^*(dw^*)$, where $\langle w^*\rangle$ is the quadratic covariation of the path $w$ which is defined $\n^*$-a.e. 
	Suppose that $S$ has L\'{e}vy-It\^{o} decomposition \eqref{eqn:PRM_subordinator} and is independent of the process $w$ whose ``law'' is $\n$. By Theorem \ref{thm:excursionChange} the ``law'' of $w^*:= \Xi(w,M,\kappa)$ is $\n^*$. 
	
	However, observe that 
	\begin{align}
	\zeta(w^*) &= \inf\{t>0: w^*(t) = 0\}= \inf\{t>0: \Xi(w,M,\kappa)(t) = 0\}\nonumber\\
	&=\inf\{t : \inf\{u:\kappa u + \int_0^\infty y M([0,u],dy) > t\} = \zeta(w)\}\nonumber\\
	&= \kappa \zeta(w) + \int_0^{\infty} y M([0,\zeta(w)],dy) = S(\zeta(w))\label{eqn:zetawStar}
	\end{align} and
	\begin{equation*}
	\langle w^*\rangle (\zeta(w^*)) = \langle w\rangle(\zeta(w)) = \zeta(w).
	\end{equation*}
	Thus,
	\begin{align*}
	\left((\zeta(w^*), \langle w^*\rangle (\zeta(w^*)))\text{  under  } \n^* (dw^*) \right) = \left( \left(S(\zeta(w)), \zeta(w)\right) \text{  under  } \mathcal{L}(S)\otimes \n(dw) \right).
	\end{align*}
	Since the law of $\zeta(w)$ under $\n(dw)$ is know, the above representation can be used to give us an explicit description of the joint law of $(\zeta(w^*), \langle w^*\rangle (\zeta(w^*)))$
	\section{The first negative excursion of duration $D$}\label{sec:firstExcursion}	
	
	Let $(g^*,d^*)$ be the interval corresponding to the first negative excursion of $B^*$ which lasts longer than $D$. In order to price a down-and-in Parisian option, we will need to understand more than just times $g^*$ and $d^*$. In fact, we must understand the joint distribution of $(E(g^*),g^*)$ and the joint distribution of $(E(d^*),d^*)$ where $E$ is the time-change as in \eqref{eqn:Edef}.

Towards this end we will use Theorem \ref{thm:pppfortc} and Corollary \ref{cor:EFromExcursion}. Let $(e_x,M_x,\gamma_x)$ be the Poisson point process defined therein with respect to the filtration $(\GG_x)_{x>0}$, say, and let $e_x^* = \Xi(e_x,M_x,\kappa)$. Now, the first negative excursion of $B^*$ lasting longer than duration $D$ corresponds to the interval $(g^*,d^*)$ defined by
	\begin{equation}\label{eqn:gstarDstar_def}
	g^* = \inf\{\hat{\tau}_{x-}: \gamma_x{e}^*_x(D)<0\}\quad \text{and} \quad d^* = \inf\{\hat{\tau}_{x}:\gamma_x {e}^*_x(D)<0\}.
	\end{equation} Since $\gamma_x e_x^*(D)<0$ if and only if $\gamma_x = -1$ and $\zeta(e_x^*)>D$, we conclude by \eqref{eqn:zetawStar} that 
	\begin{equation*}
		\{x: \gamma_x e_x^*(D)<0\} =  \left\{x: \gamma_x = -1, \kappa \zeta(e_x) + \int_0^\infty y M_x([0,\zeta(e_x)],dy) >D \right\}.
	\end{equation*}
	We now define the following subordinators
	\begin{equation*}
	\begin{split}
		 \hat{\tau}_x &= \sum_{y\le x} \zeta(e_y^*)= \sum_{y\le x} \left\{\kappa \zeta(e_y)+ \int_0^\infty z M_y([0,\zeta(e_y)],dz) \right\} ,\\
		 \htu_x &=  \sum_{y\le x} 1_{[\gamma_y = 1]} \zeta(e_y^*)=\sum_{y\le x} 1_{[\gamma_y = 1]} \left\{\kappa \zeta(e_y)+ \int_0^\infty z M_y([0,\zeta(e_y)],dz) \right\} ,\\
		 \htd_x &=\sum_{y\le x} 1_{[\gamma_y = -1]} \zeta(e_y^*)= \sum_{y\le x}1_{[\gamma_y = -1]} \left\{\kappa \zeta(e_y)+ \int_0^\infty z M_y([0,\zeta(e_y)],dz) \right\},\\
		 \tau_x &= \sum_{y\le x} \zeta(e_y),\\
		 \tu_x &= \sum_{y\le x} 1_{[\gamma_y = 1]} \zeta(e_y),\\
		 \td_x &= \sum_{y\le x} 1_{[\gamma_y = -1]} \zeta(e_y).
	\end{split}
	\end{equation*} In words, the subordinators with $\uparrow$ as a superscript represent positive excursions and the subordinators with $\downarrow$ represent negative excursions of either $B^*$ or $B$.
	Obviously $\tau_x = \tu_x + \td_x$ and $\hat{\tau}_x = \htu_x+\htd_x$. Since $\td_x$ and $\tu_x$ are constructed from the atoms of a PRM on disjoint sets, the processes $\td$ and $\tu$ are independent and similar reasoning implies the independence of $\htd$ and $\htu$. Moreover, since $P(\gamma_x = -1) = P(\gamma_x = 1) = 1/2$, we have $(\tu,\htu)\overset{d}{=} (\td,\htd)$. 
	 To keep track of the negative excursions depending on their durations, we write
	\begin{equation*}
	\begin{split}
	\htds_x &:= \sum_{y\le x} 1_{[\gamma_y = -1, \zeta(e^*_y)\le D]} \zeta(e_y^*), \qquad
	\htdl_x:= \sum_{y\le x} 1_{[\gamma_y = -1, \zeta(e^*_y)>D]} \zeta(e_y^*),\\
	\tds_x &:= \sum_{y\le x} 1_{[\gamma_y = -1, \zeta(e^*_y)\le D]} \zeta(e_y) ,\qquad 
	\tdl_x:= \sum_{y\le x} 1_{[\gamma_y = -1, \zeta(e^*_y)>D]} \zeta(e_y).
\end{split}
	\end{equation*} We remark that in the definitions of $\tds$ and $\tdl$ there is not a typo when considering the durations $\zeta(e_y^*)$ in the indicators: we want to make sure that $\htds$ and $\tds$ jump at the same times. Here the ``$-$'' superscript represent excursions $e^*_y$ of $B^*$ lasting at most duration $D$ and the ``$+$'' superscript represents the excursions lasting longer than duration $D$. Again, since the subordinators are constructed from the behavior of a PRMs on disjoint sets, $(\htdl,\tdl)$, $(\htds,\tds)$, and $(\htu,\tu)$ are three independent $\R^2$-valued L\'{e}vy processes.

	We can compute the joint Laplace transforms of $(\hat{\tau}_x,\tau_x)$ as follows:
	\begin{equation}\label{eqn:ht_t}
		\begin{split}
		\E&\left[ \exp\left\{-\lambda \hat{\tau}_x - \theta \tau_x\right\} \right] = \E\left[\exp\left\{- \left(\sum_{y\le x} (\kappa \lambda + \theta) \zeta(e_y) + \int_0^\infty \lambda z M_y([0,\zeta(e_y)], dz)\right)\right\} \right]\\
		&= \exp\left\{ - x \int_{\EE\times \M_c} \left(1-\exp\{-(\kappa\lambda+\theta) \zeta(w) - \int_0^\infty \lambda z \mu([0,\zeta(w)],dz)\}\right)    \,\n(dw)\PR(M\in d\mu)\right\}\\
		&= \exp\left\{-x \int_\EE \left(1-\exp\left\{ -\zeta(w)(\kappa\lambda +\theta + \int_0^\infty (1-e^{-\lambda z})\,\nu(dz)) \right\} \right) \,\n(dw) \right\}\\
		&= \exp\left\{-x \int_\EE \left(1-e^{-\zeta(w)(\phi(\lambda)+\theta)}\right)\,\n(dw)\right\}\\
		&= \exp\left\{ -x \int_0^\infty (1-e^{-t (\phi(\lambda)+\theta)}) \,\frac{1}{\sqrt{2\pi t^3}}\,dt \right\}\\
		&= \exp\left\{-x \sqrt{\frac{2}{\pi}} (\phi(\lambda)+\theta)\int_0^\infty e^{-(\phi(\lambda)+\theta) t} t^{-1/2}\,dt \right\}\\
		&= \exp\left\{-x \sqrt{2 (\phi(\lambda)+\theta)} \right\}.
		\end{split}
	\end{equation} In the second and third equalities above we used the exponential formulas for PRMs \cite[Section XII.1]{RY.99}. In the fifth equality we used the \eqref{eqn:easyPropsforn}. 
	Observe that the L\'{e}vy measure of the pair $(\hat{\tau},\tau)$ is the measure $\pi(ds,dt)$ which is simply
	\begin{equation*}
	\pi(A\times B) = \int_B \PR(S(t) \in A) \frac{1}{\sqrt{2\pi t^3}} \,dt.
	\end{equation*}
	Indeed, we can instead integrate with respect to $\n(dw)$ first and go from the third item in the string of equalities above to
	\begin{equation*}
	\begin{split}
	\exp&\left\{ -x \E \left[\int_0^\infty \left(1-\exp\left\{-(\kappa \lambda + \theta) t + \int_0^\infty \lambda z M([0,t],dz) \right\} \right) \frac{1}{\sqrt{2\pi t^3}}\,dt \right]\right\}\\
	&= \exp\left\{ -x \E\left[ \int_0^\infty (1-e^{-\lambda S(t) - \theta t}) \frac{1}{\sqrt{2\pi t^3}}\,dt \right]\right\}\\
	&= \exp\left\{ -x \int_{(0,\infty)^2} (1-e^{-\lambda s -\theta t})\,\pi(ds,dt)\right\}.
	\end{split}
	\end{equation*}
	
	Similarly, we see that
	\begin{equation} \label{eqn:htuandtu}
		-\frac{1}{x}\log\E\left[\exp\left\{-\lambda \htu_x - \theta \tu_x\right\} \right] =  \sqrt{\frac{1}{2} (\phi(\lambda)+\theta)} = \int_{(0,\infty)^2} (1-e^{-\lambda s -\theta t}) \frac{1}{2}\pi (ds,dt).
	\end{equation}
	It is also easy to verify that the L\'{e}vy measure of the pure-jump L\'{e}vy process $(\htdl,\tdl)$ is $\frac{1}{2}\pi^+$ with $\pi^+(A\times B) := \pi( (A\cap (D,\infty))\times B)$ and the L\'{e}vy measure of $(\htds,\tds)$ is $\frac{1}{2}\pi^-$ where $\pi^-(A\times B):= \pi( (A\cap (0,D])\times B)$. Indeed, the jumps of $(\htdl,\tdl)$ are simply the jumps of $(\htd,\td)$ for which $\htd_x - \htd_{x-}>D$ and this is precisely what the L\'{e}vy measure $\frac{1}{2}\pi^+$ captures.
	
	Define the time $J = \inf\{x: \tdl_x>0\}$ as the first jump of the subordinator $\tdl$. Since $\tdl$ and $\htdl$ have the same jump times, $J$ is also the first jump time of $\htdl$. The time $J$ is the index $x$ of the first excursion $(e^*_x,\gamma_x)$ such that $\gamma_x = -1$ and $\zeta(e_x^*) >D$. In particular, it is easy to see that $g^* = \hat{\tau}_{J-}$ and $d^* = \hat{\tau}_J$. 
	
	We now state the following lemma for ease of reference in the sequel.
	\begin{lem}\label{lem:independence}
	With the above set-up the following hold:
	\begin{enumerate}
	\item $J$ is a $\GG_x$-stopping time;
	\item $((e_{x+J},M_{x+J}, \gamma_{x+J}); x>0)$ is independent of $\GG_J$ and is equal in law to $((e_x,M_x,\gamma_x);x>0)$;
	\item Both $d^*-g^* = \htdl_J$ and $\tdl_J$ are independent of $J$; 
	\item $J$ is independent of $\tu$, $\tds$, $\htu$ and $\htds$;
	\item $E(g^*) = \tau_{J-} = \tu_J+\tds_J$ and $g^* = \hat{\tau}_{J-} = \htu_J + \htds_J$ are each sums of two independent random variables;
	\item $E(d^*) = \tau_J = \tu_J + \tds_J + \tdl_J$ and $d^* = \hat{\tau}_J = \htu_J + \htds_J + \htdl_{J}$ are each the sum of three independent random variables.
	\end{enumerate}
	\end{lem}

	\begin{proof}
		For item (1), $J$ is the first time of the $\GG_x$-progressive process $\htdl_x$ enters the open set $(D,\infty)$ and is hence a stopping time.
		
		Item (2) is the strong Markov property for Poisson point processes.
		
		Item (3) follows from Lemma XII.1.13 in \cite{RY.99}.
		
		Item (4) follows from the fact that each of the subordinators listed are independent of $(\htdl,\tdl)$. 
		
		Items (5) and (6) follow from the definitions and \eqref{eqn:hatTauandS}. In item (5) we note that $\tau_{J-} = \tu_{J-} + \tds_{J-} = \tu_J + \tds_J$ because $\tu$ and $\tds$ are independent of $\tdl$ and independent subordinators almost surely do not jump at the same time.
	\end{proof}
	
	\subsection{Distribution of $(E(g^*),g^*)$}

	It is easy to see that $\PR(J>x) = \PR(\htdl_x = 0) = e^{-x \frac{1}{2} \pi^+(\R_+\times\R_+)}$. That is $J$ is an exponential random variable, which we state as the following Lemma:
	\begin{lem}\label{lem:JisExp} 
		The variable $J$ is exponentially distributed with parameter $\mu = \frac{1}{2}\pi^+(\R_+\times \R_+)$. 
	\end{lem}
	
	We now proceed to compute the joint Laplace transform of $(E(g^*),g^*)$ by utilizing Lemmas \ref{lem:independence} and \ref{lem:JisExp}. Indeed, by equation \eqref{eqn:htuandtu} and the description of the L\'{e}vy measure of $(\htds,\tds)$ and letting $\mu$ denote the parameter of the exponential random variable $J$, we get for each $\lambda,\theta\ge 0$:
	\begin{align}
	\E&\left[\exp\{- \lambda g^* - \theta E(g^*) \} \right] = \E\left[\exp\left\{-\lambda (\htu_J+\htds_J) - \theta (\tu_J+ \tds_J) \right\} \right] \nonumber \\
	&= \int_0^\infty {\mu} e^{-\mu x} \E\left[\exp\left\{ -\lambda \htu_x -\theta \tu_x\right\} \right]  \E\left[ \exp\left\{-\lambda \htds_x-\theta \tds_x\right\} \right]\,dx \nonumber\\
	&= \mu  \int_0^\infty e^{-\mu x} e^{-x\sqrt{\frac{1}{2} (\phi(\lambda) + \theta)}} \exp\left\{-x \frac{1}{2} \int_{(0,\infty)^2}(1- e^{-\lambda s - \theta t}) \pi^-(ds,dt) \right\}\,dx	\nonumber \\
	&= \frac{\displaystyle \mu}{ \displaystyle \mu + \sqrt{\frac{1}{2} \phi(\lambda)+\frac{1}{2}\theta} + \frac{1}{2} \int_{(0,\infty)^2} (1-e^{-\lambda s-\theta t})\,\pi^{-}(ds,dt)}. \label{eqn:jointLT_for_E_gStar}
	\end{align}
	This is precisely the Laplace transform of the measure $m_1(ds,dv)$ in Theorem \ref{thm:priceTheorem}.
	
	\subsection{Joint distribution of $(E(d^*),d^*)$}
	
	Before commencing a similar computation to that in \eqref{eqn:jointLT_for_E_gStar}, we observe, by Proposition XII.1.13 in \cite{RY.99} the joint distribution is
	\begin{equation*}
	\PR(\htdl_J\in A, \tdl_J\in B) = \frac{\frac{1}{2}\pi^+(A\times B)}{\frac{1}{2}\pi^+(\R_+\times \R_+)} = \frac{1}{2\mu} \pi^+(A\times B).
	\end{equation*}
	
	Therefore,
	\begin{align}
\nonumber 	\E&\left[\exp\left\{-\lambda d^*-\theta E(d^*)\right\} \right]= \E\left[ \exp\left\{-\lambda (g^*+\htdl_J) - \theta(E(g^*)+ \tdl_J)\right\}\right]\\
\nonumber &= \E\left[ \exp\left\{-\lambda g^*  -\theta E(g^*)\right\}\right] \E\left[\exp\left\{-\lambda \htdl_J - \theta \tdl_J\right\}\right]\\
\nonumber &= \frac{\displaystyle \mu}{ \displaystyle \mu + \sqrt{\frac{1}{2} \phi(\lambda)+\frac{1}{2}\theta} + \frac{1}{2} \int_{(0,\infty)^2} (1-e^{-\lambda s-\theta t})\,\pi^{-}(ds,dt)}  \frac{1}{2\mu}\int_{(0,\infty)^2} e^{-\lambda s- \theta t} \pi^+(ds,dt)\\
&=\frac{\displaystyle \frac{1}{2}\int_{(0,\infty)^2} e^{-\lambda s- \theta t} \pi^+(ds,dt)}{ \displaystyle \mu + \sqrt{\frac{\phi(\lambda)+\theta}{2} } + \frac{1}{2} \int_{(0,\infty)^2} (1-e^{-\lambda s-\theta t})\,\pi^{-}(ds,dt)}  \label{eqn:jointLT_for_E_dstar}.
	\end{align}
This is precisely the Laplace transform of $m_2(ds,dv)$ in Theorem \ref{thm:priceTheorem}.

	\section{A Parisian option in illiquid markets}\label{sec:illiquidMarkets}
	
	As mentioned in the introduction, one attempt to overcome the inability to sell assets for different intervals of time is to introduce a time-change by the inverse of a strictly increasing suboridnator. We consider a geometric Brownian motion, $X$, satisfying \eqref{eqn:gmb} where $W$ is a Brownian motion under a measure $\QQ$ and $S$ is an independent strictly increasing subordinator. The price of the illiquid asset is then modeled by $X^*(t) = X(E(t))$ where $E$ is defined in \eqref{eqn:Edef}. The process $X^*$ can be written as
	\begin{equation*}
	X^*(t) = x \exp\left( \sigma W(E(t)) + (\mu- \frac{\sigma^2}{2})E(t)\right),
	\end{equation*} which is the unique solution to \eqref{eqn:agbm}.
	
	The incorporation of interest into this model is a bit tricky because one cannot sell the stock at anytime (it is illiquid) in order to obtain the risk free bond and so, for simplicity, we assume that the interest rate is zero, i.e. $r = 0$. See \cite{MG.12} for one approach of incorporating non-zero interest rates. As shown in \cite{Magdziarz.09} in the case where $S$ is a stable subordinator and in \cite{MG.12} in general, the market $(\Om, (\F^*_t), \F^*, \QQ)$ with $\Om = C([0,\infty),\R)$, $\F^*_t$ is $\sigma(X^*(s); s\le t)$ completed in the standard way with all the null sets is arbitrage free for finite-time intervals but incomplete.
	
	Fix a deterministic time $T> 0$. A non-unique equivalent martingale measure $\Q$ is given by
	\begin{equation}\label{eqn:dQdPFT}
	\frac{d\Q|_{\F^*_T}}{d \QQ|_{\F^*_T}} = \exp\left\{-\frac{\mu}{\sigma} W(E(T)) - \frac{\mu^2}{2\sigma^2} E(T) \right\}.
	\end{equation} Setting $Z(t) = W(t) + \frac{\mu}{\sigma}t$, the authors of \cite{MG.12} show that $X^*(t)$ and $Z^*(t) = Z(E(t))$ are $\Q$-local martingales (see a similar derivation below). Note that $(S(t);t\ge 0)$ is a family of increasing $\{\F^*_t\}_{t\ge0}$-stopping times. 
	Hence, by the Dambis-Dubins-Schwarz theorem \cite[Theorem V.1.6]{RY.99} the process $Z$ is a Brownian motion in the time-changed filtration $\F^*_{S(t)}$.
	
	The process $X^*(t) = x \exp\left\{ \sigma ( Z(E(t)) +  \frac{\sigma}{2} E(t) ) \right\}$ is a time-changed geometric Brownian motion under $\Q$. This follows from $Z$ being independent of $S$. Indeed, consider two random variables $F = F(Z(u);u\in[0,T])$ and $G = G(S(u);u\in[0,T])$ which are both $\F^*_{S(T)}$-measurable. Moreover, by Proposition VIII.1.3 in \cite{RY.99}, we have
	\begin{equation*}
	\begin{split}
	\Q(A\cap \{S(T) <\infty\}) &= \int_{A\cap \{S(T)<\infty\}} \exp\left\{ -\frac{\mu}{\sigma} Z\left(E (S(T))\right) + \frac{\mu^2}{2\sigma^2} E(S(T))\right\} \,d\QQ\\
	&= \int_A \exp\left\{-\frac{\mu}{\sigma} Z(T) + \frac{\mu^2}{\sigma^2} T\right\}\,d\QQ
	\end{split}
	\end{equation*} since $S(T)<\infty$ and $E(S(T)) = T$ hold $\QQ$ almost surely.
	 Thus we have
	\begin{equation*}
	\begin{split}
	\E_{\Q}[F G] &= \E_{\QQ} \left[ e^{-\frac{\mu}{\sigma} Z(T)+  \frac{\mu^2}{2\sigma^2}T} FG\right]\\
	&= \E_{\QQ}\left[  e^{-\frac{\mu}{\sigma} Z(T)+  \frac{\mu^2}{2\sigma^2}T} F \right] \E_{\QQ} [G]\\
	&= \E_{\Q}[F] \E_{\QQ}[G]\\
	&= \E_{\Q}[F] \E_{\QQ}[e^{-\frac{\mu}{\sigma} Z(T) + \frac{\mu^2}{2\sigma^2} T} G]\\
	&= \E_{\Q}[F] \E_{\Q}[G],
	\end{split}
	\end{equation*} where in the penultimate step we used the independence (under $\QQ$) of $Z$ and $S$ and the fact that 
	\begin{equation*}
	\E_{\QQ}[\exp(-\frac{\mu}{\sigma} Z(T))] = \E_{\QQ}[\exp(-\frac{\mu}{\sigma} W(T) - \frac{\mu^2}{\sigma^2} T)] = \exp( -\frac{1}{2} \frac{\mu^2}{\sigma^2} T).
	\end{equation*} Setting $F = 1$ and $G = \exp(-\lambda S(t))$, we can see that $S$ is a subordinator under $\Q$ with the same Laplace exponent.
	
	Now it is easy to see that the price of the down-and-in European Parisian option  with barrier $L$, duration $D$, maturation $T$ and payoff $\Phi$ in this illiquid market is indeed given by equation \eqref{eqn:CdiForTimeChange}. We will now do what was outlined in Section \ref{sec:parisianPrice} which did not involve a time-change. That is, we change measure to a measure $\PR$ under which $B(t) = Z(t)+\frac{\sigma}{2} t$ is a Brownian motion in the filtration $(\F^*_{S(t)})$.
	
	\subsection{The measure $\PR$}
	
	Define the measure $\PR$ on $\F^*_t$ by, for $A\in \F^*_t$,
	\begin{equation}\label{eqn:dQQdQ}
	\PR(A) = \int_A \exp\left\{ - \frac{\sigma}{2} Z(E(t)) - \frac{\sigma^2}{8} E(t)\right\} \,d\Q.
	\end{equation}
	First, since $Z$ is independent of $E$ under $\Q$ and $Z$ is Brownian motion under $\Q$, the measure $\PR$ is a probability measure. Indeed for any $\lambda\in \R$, disintegration gives
	\begin{equation*}
	\begin{split}
	\E_{\Q}\left[ \exp\left\{\lambda Z(E(t)) - \frac{\lambda^2}{2} E(t)\right\} \right] &= \int_0^\infty \E_\Q[\exp(\lambda Z(u)- \frac{\lambda^2}{2}u)] \,\Q(E(t)\in du)\\
	& = \int_0^\infty 1 \Q(E(t)\in du) = 1.
	\end{split}
	\end{equation*}
	
	Now, the Girsanov-Meyer theorem \cite[Chapter III, Theorem 40]{Protter.04} implies that
	\begin{equation*}
	Z(E(t)) - \left\langle Z(E(\cdot)), -\frac{\sigma}{2} Z(E(\cdot))\right\rangle(t) = Z(E(t))+ \frac{\sigma}{2}E(t)
	\end{equation*} is a $\PR$-local martingale. The Dambins-Dubins-Schwarz theorem \cite[Theorem V.1.6]{RY.99} then implies that $B = (B(t):=Z(t) + \frac{\sigma}{2} t; t\ge 0)$ is a $(\F^*_{S(t)})$-Brownian motion under $\PR$. A similar proof to that showing $Z$ is independent of $S$ under $\Q$ can be used to prove that $B$ is independent of the suborindator $S$ under $\PR$.
	
	We summarize this in the following lemma:
	\begin{lem} \label{lem:equivalentMeasures}
	
	With the above set-up the following hold.
	
	\begin{enumerate}
		\item $S = (S(t);t\ge 0)$ is an increasing family of $(\F^*_t)$-stopping times, and $W(t) = W(E(S(t)))$ is an $(\F^*_{S(t)})_{t\ge0}$ Brownian motion under $\QQ$.
		\item The measure $\Q$ defined by \eqref{eqn:dQdPFT} is equivalent to $\QQ$ on $\F^*_T$; i.e. it is locally equivalent to $\QQ$. Moreover:
		\begin{enumerate}
			\item The process $X^*$ solving \eqref{eqn:agbm} is an $\F^*_t$-local martingale under $\Q$;
			\item The process $Z(t) = W(t)+\frac{\mu}{\sigma} t$ is an $(\F^*_{S(t)})_{t\ge0}$-adapted process, which is a Brownian motion under $\Q$; 
			\item $Z^*(t) = Z(E(t))$ is an $\F^*_t$-local martingale under $\Q$; and
			\item The law of $S$ under $\Q$ is the same as the law of $S$ under $\QQ$ and it is independent of $Z$.
		\end{enumerate}
		\item The measure $\PR$ defined by \eqref{eqn:dQQdQ} is equivalent to $\Q$ on $\F^*_T$, that is it is locally equivalent to $\Q$. Moreover:
		\begin{enumerate}
			\item The process $B(t) = Z(t)+\frac{\sigma}{2} t$ is an $(\F^*_{S(t)})_{t\ge0}$-adapted process, which is a Brownian motion under $\PR$; 
			\item $B^*(t) = B(E(t))$ is an $\F^*_t$-local martingale under $\PR$; and
			\item The law of $S$ under $\PR$ is the same as the law of $S$ under $\QQ$ and it is independent of $B$.
		\end{enumerate}
	\end{enumerate}
	\end{lem}

	We also state the following corollary, which will be useful in any generalization of Theorem \ref{thm:priceTheorem} to the case where $L\neq x$.
	\begin{cor}\label{cor:strongMarkovIndependence}
	Let $T_b = \inf\{t: B^*(t) = b\}$ be the first hitting time of $b$ by the time-changed Brownian motion $B^*$. Then, under $\PR$,
	\begin{enumerate}
		\item $T_b = S(\inf\{t: B(t) = b\})$ is an $\F^*_{t}$-stopping time; and
		\item $B^*(t+T_b)-b$ is independent of $\F^*_{T_b}$ and is equal in distribution to $B^*$.
		\item $E(t+T_b) - E(T_b)$ is independent of $\F^*_{T_b}$ and is equal in distribution to $E$.
	\end{enumerate}
	\end{cor}
	\begin{proof}
	We begin with proving item 1. First note that $\F^*_t$ satisfies the usual conditions and $B$ is $\F^*_t$-progressive, so $T_b$ is a d\'{e}but time and is hence a stopping time. Note that
	\begin{equation*}
		\begin{split}
			T_b &= \inf\{t: B(E(t)) = b\}= \inf\{t: E(t) = \inf\{u: B(u) = b\}\}.
		\end{split}
	\end{equation*}	
	Let $\tilde{T}_b = \inf\{u: B(u) = b\}$ be the first hitting time of $b$ by the Brownian motion $B$. Obviously, $\tilde{T}_b$ is an $\F^*_{S(t)}$-stopping time. 
	
	Also note that $S(\tilde{T}_b) = S(\tilde{T}_b-)$ with probability 1 since $B$ is independent of $S$, and hence
	\begin{equation}\label{eqn:TBSTB}
		T_b = \inf\{t:E(t) = \tilde{T}_b\} = S(\tilde{T}_b-) = S(\tilde{T}_b).
	\end{equation} This proves item 1.

	To prove items 2 and 3, we note that $(B,S)$ is a L\'{e}vy process under $\PR$ and thus strong Markov process under $\PR$ with respect to the filtration $(\F^*_{S(t)})_{t\ge 0}$. Thus $\tilde{B} = (\tilde{B}(t);t\ge 0)$ defined by $\tilde{B}(t) = B(t+\tilde{T}_b)-b$ and $\tilde{S} = (\tilde{S}(t);t\ge 0)$ defined by $\tilde{S}(t):=S(t+\tilde{T}_b)-S(\tilde{T}_b)$ are each independent of $\F^*_{S(\tilde{T}_b)} = \F^*_{T_b}$ and are equal in law to $B$ and $S$, respectively.
	
	Now observe from equation \eqref{eqn:TBSTB}
	\begin{equation*}
		\begin{split}
		E(t+T_b) &= \inf\{u: S(u)> t+S(\tilde{T}_b)\} = \inf\{v+\tilde{T}_b: S(v+\tilde{T}_b)-S(\tilde{T}_b) > t\} \\
		&= \inf\{v: \tilde{S}(v)>t\} + \tilde{T}_b.
		\end{split}
	\end{equation*} Therefore, if $\tilde{E}$ is the right-continuous inverse of $\tilde{S}$ then
\begin{equation*}
	E(t+T_b)  = \tilde{E}(t) +\tilde{T}_b,
\end{equation*} which, together with \eqref{eqn:TBSTB}, proves item 3. But then
\begin{equation*}
	B^*(t+T_b)  = B(E(t+T_b)) = B(\tilde{E}(t)+\tilde{T}_b) = b+\tilde{B}(\tilde{E}(t)),
	\end{equation*} which proves item 2.


	\end{proof}
	
	We finish this section by giving an expression for \eqref{eqn:CdiForTimeChange} under the measure $\PR$ and the Brownian motion $B$. Before writing that down note, that $H_{L,D}^-(X^*)\le T$ if and only if $H_{b,D}^-(B^*)\le T$ where $b = \frac{1}{\sigma} \log(L/x)$. Consequently, comparing equations \eqref{eqn:CdiForTimeChange} and \eqref{eqn:dQQdQ} we arrive at
	\begin{equation}\label{eqn:downandincall}
	\begin{split}
	V_{d,i}^*&(x,T,L,D,\Phi)\\
	&=\E_{\PR}\left[ 1_{[H_{b,D}^-(B^*)\le T]} \exp\left\{\frac{\sigma}{2} B(E(T)) - \frac{\sigma^2}{8} E(T) \right\} \Phi(xe^{\sigma B(E(T))}) \right]\\
	&= \E_{\PR}\left[1_{[H_{b,D}^-(B^*)\le T]} f_{x,\sigma}(B(E(T)), E(T))  \right]
	\end{split}
	\end{equation}
	where $f_{x,\sigma}(z,s) = \exp\left( \frac{\sigma}{2} z - \frac{\sigma^2}{8} s \right) \Phi\left(xe^{\sigma z}\right)$.
	
	\subsection{The proof of Theorem \ref{thm:priceTheorem}}

	
	Before continuing, recall the excursion construction of both a Brownian motion and its time-change in Theorem \ref{thm:pppfortc} and Corollary \ref{cor:EFromExcursion}. We will use the notation in Corollary \ref{cor:EFromExcursion} along with the various subordinators $\tu, \htu$, etc. freely.
	
	We note that in Theorem \ref{thm:priceTheorem} we are pricing a Parisian option where the barrier $L$ is the initial stock price, i.e. $L = x$. By the discussion at the end of the previous section this means that $b = 0$. Moreover, since the Brownian motion $B$ starts from $0$ we have 
	\begin{equation*}
	\{H_{b,D}^-(B^*) \le T\} = \{H_{0,D}^{-} (B^*)\le T\}= \{g^*\le T-D\},
	\end{equation*} where $g^*$ is defined in \eqref{eqn:gstarDstar_def} and is studied in Section \ref{sec:firstExcursion}. Indeed, $H_{0,D}^-(B^*)$ is the first time the time-change $B^*$ spends $D$ consecutive units of time below the origin. This occurs before time $T$ if and only if the corresponding excursion starts before time $T-D$.
	We now have two situations to consider: 1) $g^*\le T-D$ but $d^*> T$; and 2) $d^*\le T$. We handle each case separately.
	
	\subsubsection{Case 1: $g^*\le T-D$ and $d^*>T$}
	
	Recall $J = \inf\{x:\gamma_x =-1, \zeta(e_x^*) >D\}$ is exponentially distributed. The following two events are equal: 
	\begin{equation*}
	\{g^*\le T-D, d^*>T\} = \{\hat{\tau}_{J-} \le T-D, \hat{\tau}_J >T \}.
	\end{equation*}

	We have
	\begin{align}
	\E&\left[1_{[g^*\le T-D, d^*> T]} f_{x,\sigma}\left(B(E(T)), E(T)\right) \right]\nonumber\\
	&= \E\left[ 1_{[g^*\le T-D, d^*>T]} e^{-\frac{\sigma^2}{8} E(g^*)} f_{x,\sigma}\left(B(E(T)), E(T)-E(g^*)\right) \right]\nonumber \\
	&= \int \PR(g^*\in ds, E(g^*)\in dv) 1_{[s\le T-D]} e^{-\frac{\sigma^2}{8}v}  \nonumber\\
	&\qquad\qquad \times  \int_{\EE} \n^*(dw^*|\zeta(w^*)>T-s) f_{x,\sigma}\left(-w^*(T-s), \langle w^*\rangle (T-s) \right). \label{eqn:u1Help1}
	\end{align}	Indeed, the first equality is simply the fact $f_{x,\sigma}(z, s) = e^{-\frac{\sigma^2}{8}s_0} f_{x,\sigma}(z,s-s_0)$. The second equality follows from conditioning the value of $(g^*,E(g^*))$, Proposition XII.1.13 in \cite{RY.99} and recognizing that conditioning on $g^* = s$ and $d^*>T$ the excursion over $(g^*,d^*)$ is distributed as $-w^*$ under $\n^*(dw^* | \zeta(w^*)> T-s).$

	Before continuing, let us discuss the one-dimensional time-marginal distribution of $(w^*(t), \langle w^*\rangle(t))$ under $\n^*(dw^* | \zeta(w^*)> t)$. By Theorem \ref{thm:excursionChange} and the same reasoning in \eqref{eqn:zetawStar}, this is equivalent to understanding
	\begin{equation*}
	\Xi(w, M, \kappa)(t) \qquad\text{under}\quad \left( \n\otimes \mathcal{L}(M)\right) \left(dw\times dM \bigg| \int_0^{\infty} y M([0,\zeta(w)]) + \kappa \zeta(w) > t\right).
	\end{equation*}
	This is also the same as
	\begin{equation*}
	\Xi(w,M,\kappa)(t) \qquad\text{under}\quad \left(\n\otimes \mathcal{L}(M)\right) \left(dw\times dM \bigg| \zeta(w)> \inf\left\{u: \int_0^\infty yM([0,u],dy) + \kappa u>t \right\}\right).
	\end{equation*}
	It is now easy to see that 
	\begin{equation*}
	\left(w^* \text{  under  } \n^*(dw^*| \zeta(w^*)>t) \right) \overset{d}{=} \left(w(E(\cdot)) \text{  under  } w\sim \n(dw| \zeta(w)> E(t)), E\text{ independent of }w \right).
	\end{equation*}

	We observe that $w(t)$ under $\n(\ \cdot \ | \zeta(w)>t)$ has law \cite[Section XII.4]{RY.99}
	\begin{equation*}
	\n(w(t)\in dz| \zeta(w)>t)= \frac{1}{\n(\zeta>t)} \sqrt{\frac{2}{\pi t^3}} z e^{-z^2/2t}\,dz = \frac{z}{t} e^{-z^2/2t}\,dz.
	\end{equation*}
	This is precisely the marginal density of a 2-dimensional Bessel process at time $t$ \cite[Section XI.1]{RY.99}. Let $Y = (Y(u);u\ge 0)$ denote a 2-dimensional Bessel process starting from $0$ independent of $E$ and fix a bounded function $f:\R^2\to \R$. By the reasoning above and writing $\E^E$ for the expectation only with respect to the process $E$ and similarly for $\E^Y$ we have
	\begin{equation}\label{eqn:u1Help2}
	\begin{split}
	\int_{\EE} &f\left(w^*(t), \langle w^*\rangle(t) \right)\,\n^*(dw^*|\zeta(w^*)>t)= \E^{E} \left[\int_{\EE}  f\left( w(E(t)), E(t)\right)  \,\n(dw | \zeta(w)>E(t))\right]\\
	&= \E^E \left[ \E^Y [ f(Y(E(t)),E(t)) ] \right] = \E[f(Y(E(t)),E(t))].
	\end{split}
	\end{equation}
	
	More generally, let $Y$ be any $2$-dimensional Bessel process starting from $Y(0)\ge 0$. Let $a(t):=a(0)+t$ be a deterministic Markov process. Observe that the 2-dimensional process $\xi (t) = (Y(t),a(t))$ is strong Markov with respect to measures $\{\mathbb P_{(z,s)}: (z,s) \in \R_{+} \times \R\}$ under which $\xi(0)=(z,s)$ $\PR_{(z,s)}$-almost surely. Denote by $\mathbb E_{(z,s)}$ its corresponding mathematical expectation. In \eqref{eqn:u1Help2}, the expectation $\mathbb E$ stands for $\mathbb E_{(0,0)}$. The process $\xi$ has generator
	\begin{equation*}
	\mathscr{L}_{\xi} f(z,s) = \left(\frac{1}{2} \frac{\partial^2}{\partial z^2} + \frac{1}{2z} \frac{\partial}{\partial z} + \frac{\partial}{\partial s}\right) f(z,s),
	\end{equation*}
and the domain contains functions $f\in C^\infty$ such that $f(z,s) = g(z)h(s)$ for smooth functions $g$ and $h$ such that $g$ vanishes in a neighborhood of zero. This can be seen from \cite[Section 2.7]{IM.74} for the Bessel coordinate $Y$ and a direct computation afterwards. In particular $f(z,s ) =f_{x,\sigma}(-z,s)$ is in the domain of $\mathscr{L}_\xi$.
	As discussed in Section \ref{sec:timeFracPDE}, we can write for any $x = (z,s)\in \R_+\times\R$
	\begin{equation*}
	u(z,s;t) := \E_{(z,s)}[f(\xi (E(t)))],
	\end{equation*} as the unique solution to
	\begin{equation}\label{eqn:u1Help3}
	\partial_t^S u = \mathscr{L}_\xi u,\qquad\text{with}\qquad u(z,s;0) = f(z,s),
	\end{equation}
	where $\partial_t^S$ is defined in \eqref{eqn:genCaputo}.
	
	Combining \eqref{eqn:u1Help1}, \eqref{eqn:u1Help2} and the time-fractional PDE discussion above, we get
	\begin{equation}\label{eqn:u1m1inproof}
	\begin{split}
	\E&\left[1_{[g^*\le T-D, d^*> T]} f_{x,\sigma}\left(B(E(T)), E(T)\right) \right]\\
	&=\int 1_{[s\le T-D]} e^{-\frac{\sigma^2}{8} v} u_1(0,0; T-s)\,m_1(ds,dv)
	\end{split}
	\end{equation} where $u_1$ solves \eqref{eqn:u1Help3} with  $f(z,s) = f_{x,\sigma}(-z,s) = \exp\left\{-\frac{\sigma}{2} z - \frac{\sigma^2}{8}s\right\}\Phi(xe^{-\sigma z})$ and $m_1(ds,dv) = \PR(g^*\in ds, E(g^*)\in dv)$ which has Laplace transform \eqref{eqn:jointLT_for_E_gStar}. This matches the definitions appearing in Theorem \ref{thm:priceTheorem}.
	
	\subsubsection{Case 2: $d^*\le T$}
	
	The second case is relatively easier because of conclusion 2 in Lemma \ref{lem:independence}. Indeed, $(e^*_{x+J}, \gamma_{x+J})_{x>0}$ is equal in law to $(e_x^*,\gamma_x)$ and is independent of $\GG_J$.  As in Theorem \ref{thm:pppfortc} and Corollary \ref{cor:EFromExcursion} let $\tilde{B}^* := \tilde{B}\circ \tilde{E}$ and $\tilde{E}$ be, respectively, the time-changed Brownian motion and the time-change constructed, from the excursion process $(e^*_{x+J}, \gamma_{x+J})_{x>0}$. Now observe that $d^* = \hat{\tau}_J$ and that the point process $(e^*_{x+J},\gamma_{x+J})_{x>0}$ is independent of both $d^*$ and $E(d^*)$ (cf. conclusion 2 from Lemma \ref{lem:independence}). By a similar argument, although less-involved, to the argument utilized in Case 1,  we have
	\begin{equation}\label{eqn:u2m2inproof}
	\begin{split}
	\E&\left[ 1_{[d^*\le T]} f_{x,\sigma}(B(E(T)), E(T)) \right] = \E\left[ 1_{[d^*\le T]} e^{-\frac{\sigma^2}{8} E(d^*)} f_{x,\sigma}(\tilde{B}(\tilde{E}(T-d^*)), \tilde{E}(T-d^*))\right]\\
	&= \int_0^\infty \PR( d^*\in ds, E(d^*)\in dv) 1_{[D\le s\le T]} e^{-\frac{\sigma^2}{8} v} \E\left[f_{x,\sigma}(\tilde{B}(\tilde{E}(T-s)), \tilde{E}(T-s))\right]\\
	&= \int 1_{[D\le s\le T]} e^{-\frac{\sigma^2}{8} v} u_2(0,0;T-s)\,m_2(ds,dv),
	\end{split}
	\end{equation}	 where $u_2$ is the unique solution to 
\begin{equation*}
	\partial_t^S u = \left( \frac{1}{2} \frac{\partial^2}{\partial z^2} + \frac{\partial}{\partial s}\right) u,\qquad u(z,s;0) = f_{x,\sigma}(z,s)
\end{equation*}and $m_2(ds,dv) = \PR(d^*\in ds, E(d^*)\in dv)$. The measure $m_2$ has Laplace transform \eqref{eqn:jointLT_for_E_dstar}. These two definitions of $u_2$ and $m_2$ match the definitions found in Theorem \ref{thm:priceTheorem}.

	The proof of Theorem \ref{thm:priceTheorem} now follows from adding together equations \eqref{eqn:u1m1inproof} and \eqref{eqn:u2m2inproof}.

\section{Consequences of the Mapping Theorem}\label{sec:occupationChange}

A major reason why we could obtain some explicit results in Section \ref{sec:basicn*}, particularly in the $\beta$-stable subordinator case, is because we can relate an excursion of $R$ of duration $T$ to an excursion of $R^*$ of duration $T^*\overset{d}{=}S(T)$. As we will see in this section, this is essentially a consequence of the mapping theorem (Lemma \ref{lem:mappingTheorem}).

We now recall briefly the set-up of Theorem \ref{thm:timeChange1}. The process $X$ is a c\`{a}dl\`{a}g process and $S$ is an independent strictly increasing subordinator. We let $X^*(t) = X({E(t)})$ denote the time-changed process. Also recall the set-up of $\nu(I; -)$ and $\nu^*(I;-)$ in and around Theorem \ref{thm:timeChange1}, along with the notation $A(v)$ and $A^*(v)$. 

\begin{proof}[Proof of Theorem \ref{thm:timeChange1}]
	
	By conditioning on $(X(t);t\in[0,\tau]) = (w(t);t\in[0,T])$, which is independent of $S$ it suffices to show the following lemma:
		\begin{lem}\label{lem:AvDeterministic}
		Let $w:[0,T]\to \R$ be Borel measurable for some $T\in[0,\infty]$. Let $\nu^*(dv)$ and $\nu(dv)$ be the occupation measures of $(w(E(t));t\in[0,S(T)])$ and $w$ respectively. Define $A^*(v) = \nu^*(-\infty, v]$ and $A(v) = \nu(-\infty, v]$.
		Then
		\begin{equation*}
			\left(A^*(v);v\in\R \right) \overset{d}{=} \left(S(A(v));v\in\R \right).
		\end{equation*}
	\end{lem}

	 \begin{figure*}
	\centering
	\begin{subfigure}[b]{0.4\textwidth}
		\centering
		\includegraphics[width=\textwidth]{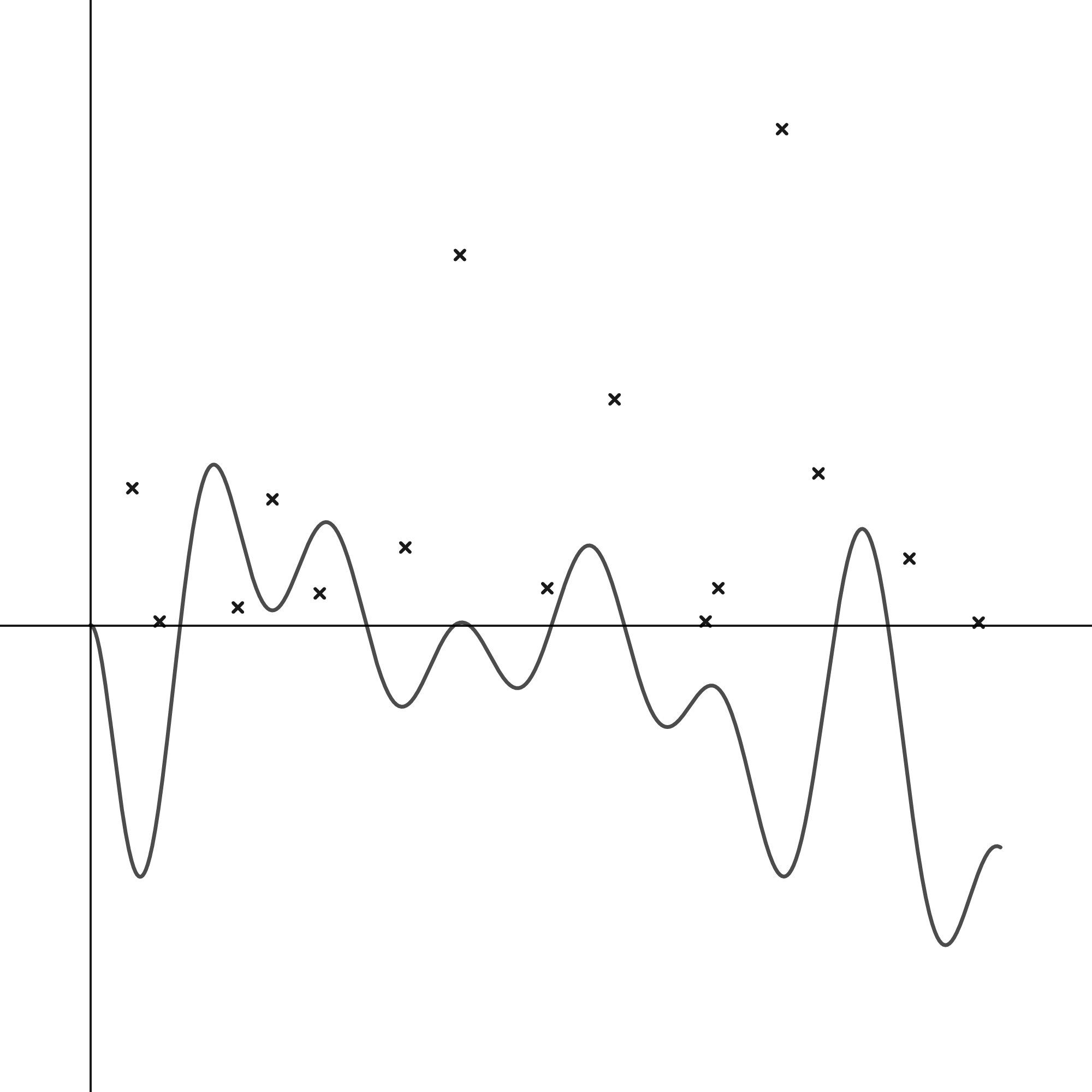} 
		\caption{}
	\end{subfigure}
	\hfill
	\begin{subfigure}[b]{0.4\textwidth}  
		\centering 
		\includegraphics[width=\textwidth]{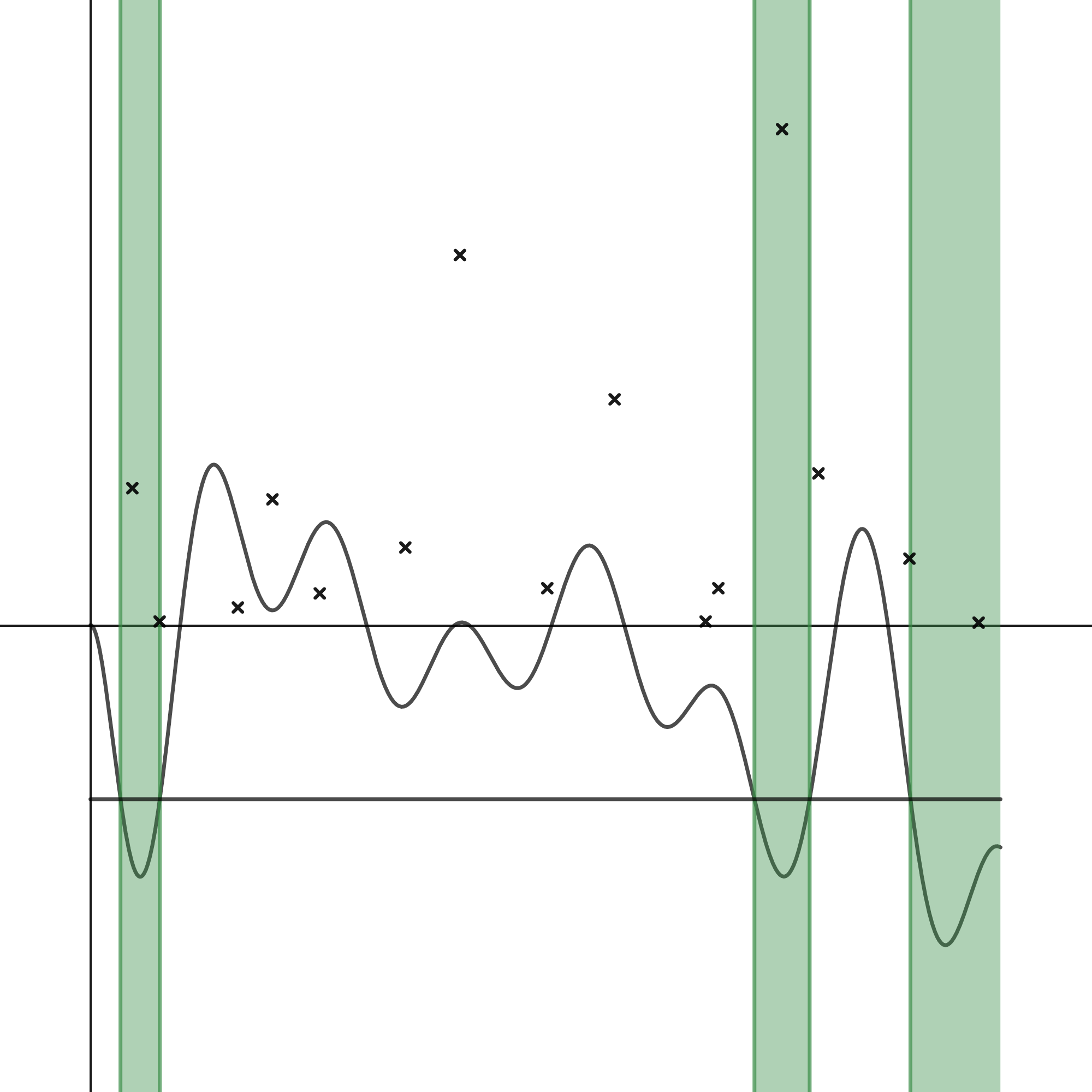} 
		\caption{}
	\end{subfigure}
	\vskip\baselineskip
	\begin{subfigure}[b]{0.4\textwidth}   
		\centering 
		\includegraphics[width=\textwidth]{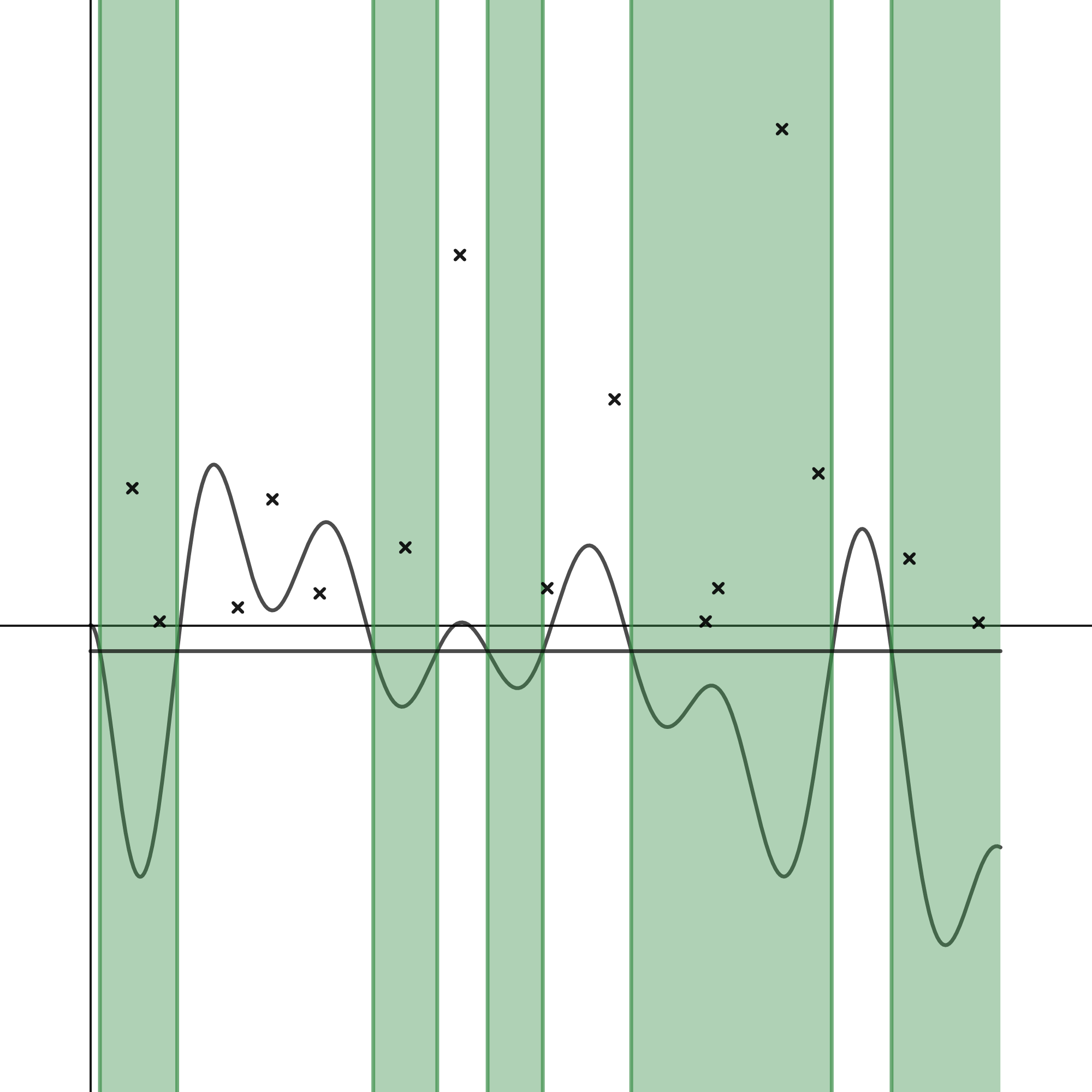}  
		\caption{}
	\end{subfigure}
	\hfill
	\begin{subfigure}[b]{0.4\textwidth}   
		\centering 
		\includegraphics[width=\textwidth]{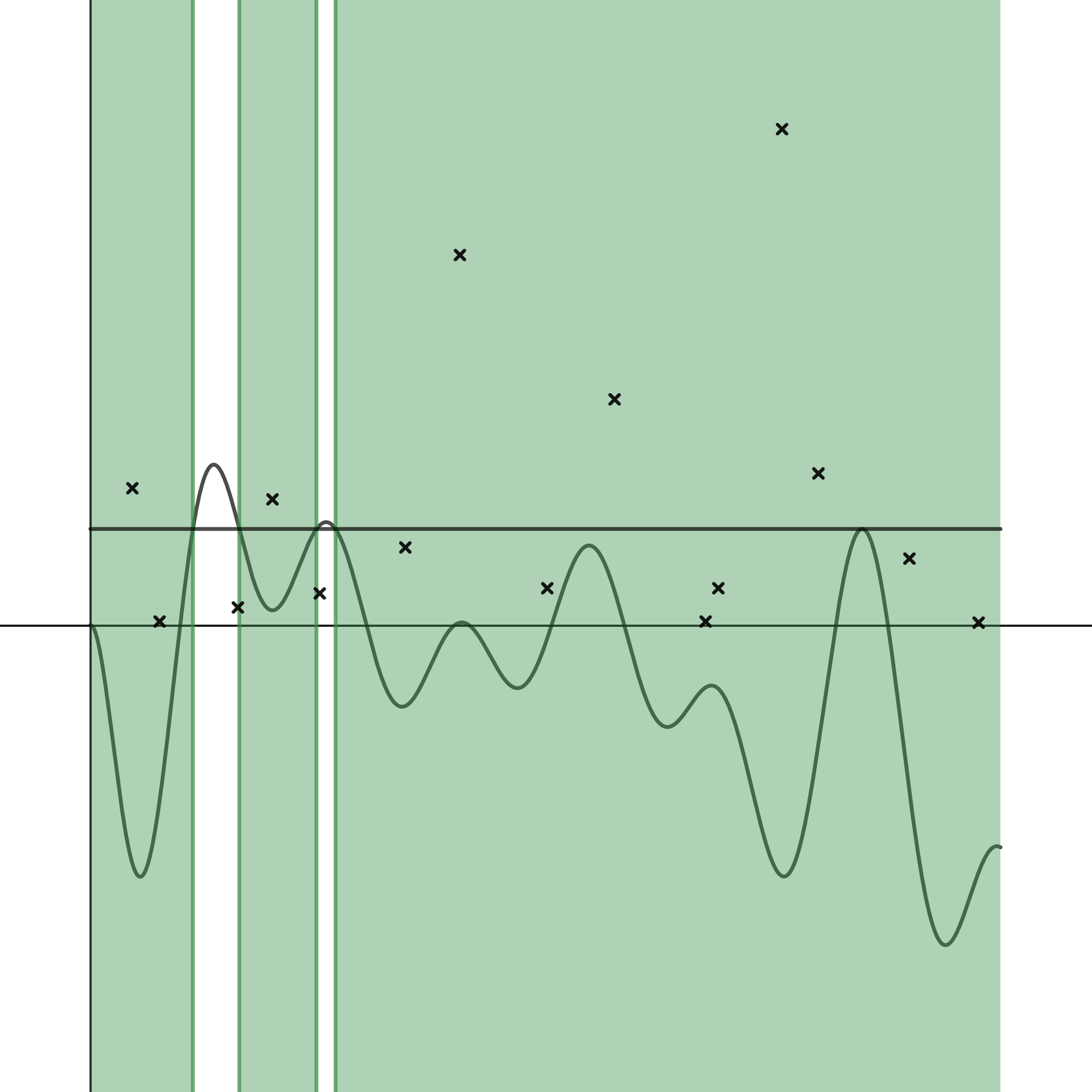}
		\caption{}
		
	\end{subfigure}
	\caption
	{\small Visual representation of how the PRM $M^{\ast}$ is constructed from a path $w$ and a PRM $M$. The atoms of the PRM $N$ are depicted as $\times$'s in the four subfigures, where the atoms of the PRM $\tilde{M}$ are depicted as happening at time $v = w(t)$ whenever the horizontal line is at level $v$. Said another way, the $\times$'s contained in the green shaded area are atoms of $\tilde{M}|_{(-\infty,v]\times (0,\infty)}$.}.
	\label{fig:prmTimeChange}
\end{figure*}
	
\end{proof}
	
	\begin{proof}[Proof of Lemma \ref{lem:AvDeterministic}]
 We note that
	\begin{equation*}
		\int_0^T g(w(t))\,dt = \int_\R g(a)\,\nu(da),\qquad\forall g\text{ bounded and measurable}.
	\end{equation*}
	
	Using the L\'{e}vy-It\^{o} decomposition of $S$, we now observe the simple description of the stochastic (Stieljes) integral with respect to $S$
	\begin{align*}
		\int_0^T g(w(t))\,dS(t) &= \kappa \int_0^T g(w(t))\,dt + \int_0^T \int_{(0,\infty)} g(w(t))\,y M(dt,dy)\\
		&= \kappa \int_\R g(x)\,\nu(dx) + \int_\R \int_{(0,\infty)} g(x)y M^{\ast}(dx,dy),
	\end{align*}
	where $M$ is a PRM with intensity $\Leb\otimes \pi$ and $M^{\ast}$ is the pushforward measure
	\begin{equation*}
		M^{\ast} = (w\times \operatorname{id}_{(0,\infty)})_\# M\quad\text{under}\qquad(w\times \operatorname{id}_{(0,\infty)})(t,y) = (w(t),y).
	\end{equation*} 
	By the mapping theorem, Lemma \ref{lem:mappingTheorem}, the measure $M^{\ast}$ is a Poisson random measure on $\R\times (0,\infty)$ with intensity $\nu\otimes\pi = (w\times \operatorname{id}_{(0,\infty)})_\# (\Leb|_{[0,T]}\otimes \pi)$.
	
	By the change of variable for Stieljes integrals \cite{FT.2012}, we can write
	\begin{equation*}
		\int_0^T g(w(t))\,dS(t) = \int_0^{S(T)}g(w(E(u)))\,du,
	\end{equation*} where $E$ is the right-continuous inverse of $S$ as in \eqref{eqn:Edef}. Therefore, the occupation measure $\nu^*$ of the time-changed path $w^*(u) = w(E(u))$ on the random time-interval $[0,S(T)]$ satisfies
	\begin{equation}\label{eqn:PRMrepNu*}
		\int_\R g(x)\,\nu^\ast(da) = \kappa \int_\R g(x)\,\nu(dx) + \int_\R \int_{(0,\infty)} g(x)y \,M^\ast (dx,dy),\qquad \forall g \in C_c(\R).
	\end{equation}
	
	Equation \eqref{eqn:PRMrepNu*} uniquely determines the occupation measure $\nu^*$ and hence $A^*$. In particular, we get
	\begin{equation*}
		A^*(v) = \kappa\nu(-\infty,v] + \int_{(0,\infty)} y \, M^* ((-\infty,v],dy).
	\end{equation*}
	Recall that $M^*(dx,dy)$ is a Poisson random measure on $\R\times (0,\infty)$ with intensity measure $\nu\otimes \pi$ but so is $\tilde{M}$ defined by $\tilde{M}((-\infty , v]\times B) = M((0,A(v)]\times B)$ as can be easily verified. It follows
	\begin{equation*}
		\left(A^*(v);v\ge 0\right) \overset{d}{=} \left( \kappa A(v) + \int_{(0,\infty)}y \tilde{M}((0,A(v)], dy; v\ge 0 \right)\overset{d}{=} \left(S(A(v));v\ge \right),
	\end{equation*}
	as desired.
	\end{proof}

\section{A Ray-Knight Theorem}\label{sec:rayKnight}
The classical Ray-Knight theorems \cite[Chapter XI]{RY.99} describe the local time of a Brownian motion stopped at certain random times in terms of a piecewise concatenation of squared Bessel processes of dimensions 0, 2, and 4. Said in another way, the density of the occupation measure of Brownian motion at certain random time intervals is a concatenation of continuous state branching processes with immigration. 
A continuous state branching process with immigration ($\CBI$ process) $Z = (Z(v);v\ge 0)$ is a strong Markov process on $[0,\infty]$ uniquely characterized by two L\'{e}vy processes $Y^{(1)}$ and $Y^{(2)}$ where $Y^{(1)}$ has no negative jumps (i.e. is spectrally positive) and $Y^{(2)}$ is a subordinator independent of $Y^{(1)}$ \cite{KW.71}. Perhaps the simplest description of these processes comes from the Lamperti transform \cite{CLU.09,Lamperti.67} as generalized by Cabellero, P\'{e}rez Garmendia and Uribe Bravo in \cite{CPU.13}: A CBI process started at $x$ with branching mechanism $\psi$ and immigration rate $\hp$ is the law of the unique c\`{a}dl\`{a}g solution to
\begin{equation}\label{eqn:lamperti}
Z(v) = x + Y^{(1)}\left(\int_0^v Z(u)\,du\right) + Y^{(2)}(v),\qquad v\ge 0
\end{equation} where $Y^{(1)}$ and $Y^{(2)}$ are L\'{e}vy processes with Laplace transforms
\begin{equation*}
\E[e^{-\lambda Y^{(1)}(t)}] = \exp\left\{t\psi(\lambda)\right\}\qquad \text{and}\qquad \E[e^{-\lambda Y^{(2)}(t)}] = \exp\left\{-t\hp(\lambda)\right\}.
\end{equation*} We denote by $\CBI_x(\psi,\hp)$ the law of a such a process. The choice of the symbol $\hp$ is to distinguish the subordinator $Y^{(2)}$ from $S$.

Let us now state the main result in this section.

\begin{thm}\label{thm:1d}
Suppose that $S$ is a strictly increasing subordinator with Laplace exponent $\phi$ with inverse $E$ in \eqref{eqn:Edef}. Let $X(t) = x + B(t) - \alpha t$ be an independent Brownian motion with negative drift started from $x\ge 0$. For $X^*(t) = X(E(t))$ and let $g:\R_+\to\R_+$ be a bounded Borel measurable function with compact support contained in $[0,a]$. Then
\begin{equation*}
\E \left[\exp\left\{-\int_0^{\tau^*} g(X^*(t))\,dt\right\} \right] = \exp \left\{ - 2\int_0^x u(s)\,ds\right\}, \qquad \tau^* = \inf\{t: X^*(t) = 0\}
\end{equation*} where $u$ is the unique non-negative solution to
\begin{equation}\label{eqn:udef}
u(r)+ 2\int_r^a (u(s)^2 + \alpha u(s))\,ds = \int_r^a \phi(g(s))\,ds
\end{equation}

\end{thm} 

The proof of Theorem \ref{thm:1d} below 
requires the characterization on positive integral functionals of CBI processes described in, for example, \cite{Li.20}. Similar, although not as clean, representations should exist when $X$ is a Bessel processes of dimension $\delta\ge 2$. This is because of the explicit representations of the occupation measures of these processes in, for example, see \cite{Yor.91} and \cite[Chapter 4]{MY.08}. We suspect similar expressions should be possible when $X$ is a transient regular diffusion \cite{MR.03}. 

\subsection{Time-changed subordinators as integrators}

Because of the distributional identity in Theorem \ref{thm:timeChange1}, we will focus on the processes $\beta = (\beta(v);v\ge 0)$ defined by 
\begin{equation*}
	\beta(v) = S\left({\int_0^v Z(u)\,du}\right) ,
\end{equation*} where $S$ is a strictly increasing subordinator with Laplace exponent $\phi$ and $Z\sim\CBI_x(\psi,\hp)$ is independent. We will denote the law of $\beta(v)$ on $\D$ by $\TIS_x(\phi;\psi,\hp)$.
Clearly, $\beta$ is non-decreasing since $S$ is strictly increasing and CBI processes are non-negative. Therefore, we can view $d\beta(v)$ as a random measure on $\R$. We will compute the Laplace transforms of a wide class of functionals involving the random measure $d\beta(v)$ explicitly.

We start by describing the Laplace transform of integral functionals of $\CBI$ processes. It will be convenient to start CBI processes at an arbitrary positive time $r\ge 0$. We will denote the law of such a process by $\CBI_{r,x}(\psi,\hp)$. That is $(Z(t); t\ge r)\sim\CBI_{r,x}(\psi,\hp)$ if and only if $(Z(t-r); t\ge r)\sim \CBI_x(\psi,\hp)$. Throughout this section we fix some branching mechanism $\psi$ and immigration rate $\hp$, and we will denote by $\PR_{r,x}$ (resp. $\E_{r,x}$) the law of (resp. the expectation with respect to) a $\CBI_{r,x}(\psi,\hp)$ process $Z$.

We state the following lemma and outline a proof.
\begin{lem}\label{lem:LiLemma}
	Let $\mu$ be a Radon measure and let $\lambda$ be a bounded non-negative Borel measurable function. Then
	\begin{equation}\label{eqn:ltOfCBI}
	\E_{r,x} \left[\exp\left\{-\int_r^t \lambda(s) Z(s)\,\mu(ds) \right\} \right] = \exp\left\{- xu(r) - \int_r^t \hp(u(s))\,ds \right\},
	\end{equation} where $u:[0,t]\to \R_+$ is the unique bounded non-negative solution to
	\begin{equation*}
	u(r)+ \int_r^t \psi(u(s))\,ds = \int_r^t \lambda(s)\mu(ds) \qquad r\in[0,t].
	\end{equation*}
\end{lem}

 In the case $\hp = 0$, the above lemma can be found in \cite{Li.20} where it is proved by approximating $\lambda(s)\mu(ds)$ by discrete measures. This lemma follows from the same approximation procedure. We just establish the initial lemma, and omit the rest of the proof.
\begin{lem}\label{lem:11}
Let $\lambda_j \ge 0$ and let $0\le t_1<\dotsm< t_n<\infty$. Then
\begin{equation*}
\E_{r,x}\left[ \exp\left\{-\sum_{j=1}^n \lambda_j Z({t_j}) 1_{[r\le t_j]}\right\} \right] = \exp\left\{ -x u(r) - \int_r^{t_n} \hp(u(s))\,ds \right\} ,
\end{equation*} where $u$ is the unique bounded non-negative solution to
\begin{equation*}
u(r) + \int_r^t \psi(u(s))\,ds = \sum_{j=1}^n \lambda_j 1_{[r\le t_j]}.
\end{equation*}
\end{lem}
\begin{proof}
	The follow the approach in \cite{Li.20} and use the Markov property and induction. This holds for $n = 1$ \cite{KW.71}. By induction we suppose that this holds for any collection of $n-1$ many $\lambda_j$'s and $t_j$'s. Since any $t_j\in[0,r)$ play no role in either the expectation nor the desired claim, we can without loss of generality assume $0\le r\le t_1<t_2<\dotsm<t_n$.

	We have by the strong Markov property for $Z$ and the induction hypothesis
	\begin{align*}
	\E_{r,x}\left[\exp\left\{-\sum_{j=1}^n \lambda_j Z(t_j)  \right\} \right] &= \E_{r,x}\left[ \E_{r,x}\left[ \exp\left\{-\lambda_1Z(t_1) - \sum_{j=2}^n \lambda_j Z(t_j) \right\} \Bigg| Z(t_1) \right]\right]\\
	&= \E_{r,x} \left[ e^{-\lambda_1 Z(t_1)} \E_{r,x} \left[ \exp \left\{-\sum_{j=2}^n \lambda_j Z(t_j) \right\} \Bigg| Z(t_1)\right] \right]\\
	&= \E_{r,x} \left[ e^{-\lambda_1 Z(t_1)} \E_{t_1, Z(t_1)} \left[\exp\left\{-\sum_{j=2}^n \lambda_j Z(t_j) \right\} \right] \right]\\
	&= \E_{r,x} \left[\exp\left\{-\lambda_1 Z(t_1) - w(t_1)Z(t_1) - \int_{t_1}^{t_n} \hp(w(s))\,ds\right\} \right],
	\end{align*} where $w$ is the unique non-negative solution to
	\begin{equation*}
	w(r)+ \int_r^t \psi(w(s))\,ds = \sum_{j=2}^n \lambda_j 1_{[r\le t_j]},\qquad r\in [0,t_n].
	\end{equation*}
	
	However, 
	\begin{equation*}
	\begin{split}
	\E_{r,x} &\left[\exp\left\{-\lambda_1 Z(t_1) - w(t_1)Z(t_1) - \int_{t_1}^{t_n} \hp(w(s))\,ds\right\} \right] \\
	&= \exp\left\{ - x\tilde{u}(r)- \int_r^{t_1}\hp(\tilde{u}(s))\,ds - \int_{t_1}^{t_n} \hp(w(s))\,ds \right\},
	\end{split}
	\end{equation*} where $\tilde{u}$ is the unique bounded non-negative solution to	
	\begin{equation*}
	\tilde{u}(r)+ \int_r^{t_1} \psi(\tilde{u}(s))\,ds = (\lambda_1 + w(t_1)) 1_{[r\le t_1]},\qquad r\in [0,t_1]. 
	\end{equation*} It is easy to check that the desired claim holds with
	\begin{equation*}
	u(s) = \left\{ 
	\begin{array}{ll}
	\tilde{u}(s) &: s\in[0,t_1]\\
	w(s) &: s\in[t_1,t_n]
	\end{array}
	\right..
	\end{equation*} 
\end{proof}

Now to describe the process $\beta(v)$ as an integrator, we will set $V = (V_t;t\ge 0)$ as the right-continuous inverse of $v\mapsto \int_0^v Z(u)\,du$. That is
\begin{equation*}
V_t = \inf\left\{v: \int_0^v Z(u)\,du>t\right\}.
\end{equation*}

A change-of-variable formula for Stieljes measures \cite[Proposition 0.4.10]{RY.99} yields the following lemma, the proof of which is omitted.
\begin{lem}\label{lem:extraLem}
	Suppose that $\hp$ is the Laplace transform of a strictly increasing subordinator. Then
	\begin{equation}\label{eqn:changeOfVariable1}
	\int_0^\infty g(v)\,d\beta(v) {=} \int_0^\infty g(V_t)\,dS(t),\qquad \forall \text{ bounded Borel $g$ with compact support}.
	\end{equation}
\end{lem}

If the equality in \eqref{eqn:changeOfVariable1} holds, we can use Lemma \ref{lem:LiLemma} and write for each bounded Borel measurable $g$ with compact support
\begin{align*}
\E&\left[\exp \left\{- \int_0^\infty  g(v)\,d\beta(v) \right\} \right]
= \E \left[ \exp \left\{- \int_0^\infty  g(V_t)\,dS(t) \right\}\right]\\
&= \E \left[\E\left[\exp\left\{-\int_0^\infty g(V_t)\,dS(t) \right\} \Bigg| (Z(t);t\ge 0)\right] \right]\\
&= \E\left[\exp\left\{-\int_0^\infty \phi(g(V_t))\,dt \right\} \right]\\
&= \E\left[\exp\left\{-\int_0^\infty \phi(g(t)) Z(t)\,dt \right\} \right]\\
&= \exp\left\{-x u(0) - \int_0^t \hp (u(s))\,ds \right\},
\end{align*} where $u$ is the unique bounded non-negative solution to
\begin{equation*}
u(r)+ \int_r^\infty  \psi(u(s))\,ds = \int_r^\infty \phi(g(s))\,ds.
\end{equation*} For the Laplace transform of stochastic Stieljes integrals involving a subordinator as the integrator - as was used in third equality above - we used the exponential martingale formula as found in, for example, see Section 5.2 of Applebaum's monograph \cite{Applebaum.09}. In the fourth equality we used the change-of-variable for Stieljes measures again.

Combined with Lemma \ref{lem:extraLem}, we have proved the following:
\begin{thm} \label{thm:ltOfIntegralsOfCBI}
	Suppose that $g$ is a bounded Borel measurable function with compact support. Let $\beta\sim \TIS_x(\phi;\psi,\hp)$, where $\hp$ is either zero or the Laplace transform of a strictly increasing L\'{e}vy process.
	
	Then
	\begin{equation*}
	\E\left[ \exp\left\{-\int_0^\infty g(v)\,d\beta(v) \right\}\right] = \exp\left\{-xu(0)-\int_0^\infty \hp(u(s))\,ds \right\},
	\end{equation*} where $u$ is the unique non-negative solution to 
	\begin{equation*}
	u(r)+\int_r^\infty \psi(u(s))\,ds = \int_r^\infty \phi\circ g(s)\,ds.
	\end{equation*}
\end{thm}

\subsection{Proof of Theorem \ref{thm:1d}}

\begin{proof}[Proof of Theorem \ref{thm:1d}]
Throughout this section we suppose that $X  = (X(t);t\ge 0)$ is a Brownian motion with negative drift started from level $x\ge 0$; i.e. $X(t) = x+ B(t) -\alpha t$ for some $\alpha\ge 0$. Let $\tau = \inf\{t: X(t) = 0\}$. A result of Warren \cite{Warren.97} states the occupation measure of $X$ on the interval $[0,\tau]$ satisfies
\begin{equation}\label{eqn:occupationTime}
\int_0^\tau g(X(s))\,ds = \int_0^\infty g(v) L(v)\,dv,
\end{equation} where $L(v)$ is a weak solution to the stochastic differential equation
\begin{equation*}
L(v) = 2(x\wedge v) + 2\int_0^v \sqrt{L(y)}\,dW(y) - 2\alpha \int_0^v L(y)\,dy.
\end{equation*} The process $L$ is a concatenation of two Markov process:
\begin{equation}\label{eqn:LinTwoZs}
L(v) = \left\{
\begin{array}{ll}
Z^{(1)}(v) &: v\in[0,x] \\
Z^{(2)}(v) &: v\ge x
\end{array}
 \right.
\end{equation} where $(Z^{(1)}(v); v\ge 0)$ is a $\CBI_0(\psi,\hp)$ process and, conditionally on $Z^{(1)}(x) = y$, the process $(Z^{(2)}(v); v\ge x)\sim \CBI_{x,y}(\psi,0)$ with
\begin{equation*}
\psi(\lambda) = 2 \lambda^2 + 2\alpha \lambda,\qquad \hp(\lambda) = 2\lambda.
\end{equation*} 

We can write by conditioning on the paths of $X$
\begin{equation*}
	\E\left[\exp(-\int_0^{\tau^*} g(X^*(s))\,ds) \right] = \E\left[\exp(-\int_0^{\tau} g(X(t))\,dS(t) )\right] = \E\left[\exp(-\int_0^\tau \phi\circ g(X(t))\,dt ) \right].
\end{equation*}
However, almost surely
\begin{equation*}
\int_0^\tau h(X(s))\,ds  = \int_0^\infty  h(v)\, L(v)\,da, \qquad \forall h\ge 0 \quad \text{Borel}.
\end{equation*} 
The rest of the proof follows as in Lemma \ref{lem:LiLemma}. Indeed, note that $h(v) = \phi(g(v))$ has support contained in $[0,a]$ since $\phi(0) = 0$. Letting $\F_s = \sigma(L(u); u\le s)$ and using the decomposition in \eqref{eqn:LinTwoZs} in terms of two Markov processes we get:
\begin{equation*}
	\begin{split}
		\E&\left[ \exp(-\int_0^\infty h(v)\, L(v)\,dv)\right] = \E\left[\exp(-\int_0^x h(v)L(v)\,dv) \E\left[ \exp(-\int_x^\infty h(v) L(v)\,dv) \bigg| \F_x\right] \right]\\
		&= \E_{0,0} \left[\exp(-\int_0^x h(v) Z^{(1)}(v)\,dv) \E_{x,Z^{(1)}(x)} \left[\exp\left(-\int_x^\infty h(v) Z^{(2)} (v)\,dv \right) \right] \right]\\
		&= \E_{0,0}\left[ \exp\left\{ -\int_0^x h(v) Z^{(1)}(v)\,dv - Z^{(1)}(x) w(x)\right\} \right]
	\end{split}
\end{equation*}
where $w$ is the unique bounded solution to
\begin{equation*}
	w(r) + \int_r^a (2 w(s)^2+ 2\alpha w(s))\,ds = \int_r^a h(s)\,ds.
\end{equation*}
Now using a similar argument to proof of Lemma \ref{lem:11}, we can write:
\begin{equation*}
\E_{0,0}\left[ \exp\left\{ -\int_0^x h(v) Z^{(1)}(v)\,dv - Z^{(1)}(x) w(x)\right\} \right] = \exp\left\{ -\int_0^x 2 u(s)\,ds\right\}
\end{equation*} where $u$ solves \eqref{eqn:udef}. 
\end{proof}

\section*{Acknowledgements}

DC is partially supported by NSF grant DMS-1444084 and NSF grant DMS-205223.


\bibliographystyle{imsart-number} 



\end{document}